\documentclass[oneside,doublespace]{amsart}

\address{Department of Mathematics, Graduate School of Science, Kyoto University, Kitashirakawa-Oiwake-cho, Sakyo-ku, Japan}
\email{sanda.fumihiko.5s@kyoto-u.ac.jp}

   \usepackage{amsthm}
   \usepackage{amsmath}
   \usepackage{amssymb}
   \usepackage{amsfonts}
   \usepackage{amscd}
   \usepackage{mathrsfs}  
   \usepackage{bm} 
   \makeatletter
    
    \@addtoreset{equation}{section}
  \makeatother   
    \newtheorem{theorem}{Theorem}[section]
  \newtheorem{definition}[theorem]{Definition}
  \newtheorem{remark}[theorem]{Remark} 
  \newtheorem{lemma}[theorem]{Lemma}
  \newtheorem{proposition}[theorem]{Proposition}

  \newtheorem{assumption}[theorem]{Assumption}

  \newtheorem{condition}[theorem]{Condition}
 \newcommand{\ainf}{A_\infty}
 \newcommand{\A}{\mathscr{A}}
 \newcommand{\B}{\mathscr{B}} 
 \newcommand{\C}{\mathbb{C}}
 \newcommand{\LL}{\mathcal{L}}
 \newcommand{\Q}{\mathbb{Q}}
 \newcommand{\p}{\mathfrak{p}}
 \newcommand{\q}{\mathfrak{q}}
 \newcommand{\R}{\mathbb{R}}
 \newcommand{\X}{\mathbb{X}}
 \newcommand{\Y}{\mathscr{Y}}
 \newcommand{\Z}{\mathbb{Z}}
 \newcommand{\ke}{\mathbb{K}}
 \newcommand{\enm}{m^\mathscr{A}}
 \newcommand{\ob}{\mathrm{Ob}}
 \newcommand{\Hom}[3]{\mathrm{hom}^\bullet_{#1}(#2,#3)}
 
 \newcommand{\coh}[2]{H^\bullet(#1; #2)}

 \newcommand{\uniho}{1_\mathrm{Hoch}}
 \newcommand{\muk}[2]{\langle #1, #2 \rangle_\mathrm{Muk}}
 \newcommand{\cyc}[2]{\langle #1, #2 \rangle_\mathrm{cyc}}
 \newcommand{\tr}{\mathrm{tr}}
 \newcommand{\str}{\mathrm{str}}
 \newcommand{\im}{\mathrm{Im}}

 \newcommand{\can}{\mathrm{can}}
 \newcommand{\pd}[3]{\langle #1, #2 \rangle_{#3}}
 \newcommand{\Fuk}{\mathrm{Fuk}}
 \newcommand{\fuk}[1]{\mathrm{Fuk}(X)^{#1}}
 \newcommand{\qcoh}{QH^\bullet(X)}
 \newcommand{\crit}{\mathrm{Crit}(X^\vee)}
 \newcommand{\mc}{H^1(L; \Lambda_{=0})}
 \newcommand{\gw}{\mathrm{GW}}

\begin{document}
  \title{Computation of quantum cohomology from Fukaya categories}
  \author{Fumihiko Sanda} 
  \date{\today}
\begin{abstract}
Assume the existence of a Fukaya category $\mathrm{Fuk}(X)$
of a compact symplectic manifold $X$ with some expected properties. 
In this paper, we show $\mathscr{A} \subset \mathrm{Fuk}(X)$ split generates a summand
$\mathrm{Fuk}(X)_e \subset \mathrm{Fuk}(X)$ corresponding to an idempotent
$e \in QH^\bullet(X)$ if the Mukai pairing of $\mathscr{A}$ is perfect.
Moreover we show $HH^\bullet(\mathscr{A}) \cong QH^\bullet(X) e$.   
As an application we compute the quantum cohomology and the Fukaya category of a blow-up of
$\mathbb{C} P^2$ at four points with a monotone symplectic structure.   
\end{abstract}
\maketitle
   \setcounter{tocdepth}{1}
   \tableofcontents
\section{Introduction}
\subsection{The main result}
Let $X$ be a compact symplectic manifold of dimension $2n$ and 
$\LL$ be a finite set of compact oriented spin weakly unobstructed Lagrangian submanifolds
such that $L_1 \pitchfork L_2$ for $L_1 \neq L_2 \in \LL$.
In this setting, Abouzaid-Fukaya-Oh-Ohta-Ono \cite{afooo} have announced a construction of a Fukaya category $\fuk{}$ (see Assumption \ref{many}). 
This is a cyclic $\ainf$ category defined over the universal Novikov field $\Lambda$ whose objects are $L \in \LL$ with Maurer-Cartan elements $b$. 
The morphisms of the cohomology category $H^\bullet(\fuk{})$ are defined by the Lagrangian intersection Floer theory \cite{fooo1}. 
Moreover, they have constructed open-closed maps
\[\p : HH_\bullet(\fuk{}) \to QH^{\bullet+n}(X),\ \ \q : \qcoh \to HH^\bullet(\fuk{})\]
with some properties (see Assumption \ref{oc}), here $HH_\bullet(\fuk{})$ (resp. $HH^\bullet(\fuk{})$) denotes the Hochschild homology (resp. cohomology) of $\fuk{}$ and $\qcoh$ denotes the small quantum cohomology of $X$.
Let $e$ be an idempotent of $\qcoh$. 
Using the morphism $\q$, we can define a subcategory $\fuk{}_e \subset \fuk{} $
corresponding to $e$ (see \S \ref{decomp}).
The main result of this paper is a version of the split generation criterion of Fukaya categories \cite{afooo}, \cite{abogeo}.
Let $\A \subset \fuk{}$ be a subcategory. 
Recall that $HH_\bullet(\A)$ is equipped with a canonical pairing called the Mukai pairing \cite {shefor}, \cite{shkhir}.
Note that $HH_\bullet(\A)$ is finite dimensional and the Mukai pairing is non-degenerate if $\A$ is smooth.
\begin{theorem}[Theorem \ref{aut}]
Assume the existence of $\fuk{}$ and open-closed maps $\p, \q.$
Let $\A$ be a subcategory of $\fuk{}$. Suppose that the Mukai pairing on $HH_\bullet(\A)$ is perfect.
Then there exists an idempotent $e$ such that $H^\bullet(\A)$ split generates the derived category $D^\pi \fuk{}_e$.
Moreover, $\p$ induces an isomorphism $HH_\bullet(\A) \cong QH^{\bullet+n}(X)e$ and $\q$ induces a ring isomorphism $\qcoh e \cong HH^\bullet(\A)$.
\end{theorem}
\begin{remark}
A similar statement has independently been obtained by Ganatra \cite{ganaut}. 
\end{remark}
\subsection{Applications}
As an application of the main result, we construct field factors of $\qcoh$ from the existence of certain Lagrangian submanifolds. 

We first consider an oriented Lagrangian torus with a spin structure $L$ such that all non-constant holomorphic disks bounded by $L$ have Maslov indices at least two (Condition \ref{positive}).
By counting holomorphic disks with Maslov index two, we can define a potential function $W_L$.
A critical point of $W_L$ gives a Maurer-Cartan element $b$ such that
the Floer cohomology $HF^\bullet(L, b)$ is not equal to zero. 
\begin{proposition}[Proposition \ref{torus} and Proposition \ref{tori}]
Let $L \in \LL$ be a torus with a spin structure which satisfies Condition \ref{positive}.
Let $b$ be a Maurer-Cartan element corresponding to a non-degenerate critical point of $W_L.$
Then  there exists an idempotent $e$ with $\qcoh e \cong \Lambda.$
Moreover, if $e'$ is an idempotent corresponding to another non-degenerate critical point, then the quantum product $e*e'$ is equal to zero. 
\end{proposition}
We next consider an even dimensional oriented rational homology Lagrangian sphere with a spin structure. 
We note that such a Lagrangian is weakly unobstructed and $0$ is a unique Maurer-Cartan element \cite{fooo1}.

\begin{proposition}[Proposition \ref{sphere}]
Let $L \in \LL$ be an even dimensional oriented spin rational homology sphere with the Maurer-Cartan element $b=0$.
Suppose that $HF^\bullet(L,0) \cong \Lambda \times \Lambda$ as a ring. 
Then there exist idempotents $e_L^\pm$ such that $\qcoh e^\pm_L \cong \Lambda$ and $e^+_L*e^-_L=0.$
Moreover these idempotents are eigenvectors of the operator $c_1*$ with the same eigenvalue, here $c_1$ denotes the first Chern class of the Tangent bundle of $X$.
\end{proposition}
\subsection{Examples}
As a first example, we consider a compact toric Fano manidold $X$.
In \cite{ritcir}, Ritter has proved that the torus fibers of $X$ split generate the Fukaya category of $X$
if the potential functions of the torus  fibers are Morse.
We give another proof of this statement (see Theorem \ref{Toric}).

As a next example we consider a blow-up of $\C P^2$ at four points with a monotone symplectic structure. 
By a result of Li-Wu \cite{lilag}, we can construct many Lagrangian spheres.
Moreover, Pascaleff-Tonkonog \cite{paswal} construct a monotone Lagrangian torus and compute its potential function. 
Using these results, we prove the following:
\begin{theorem}[Theorem \ref{blow-up}]
Let $X$ be a blow-up of $\C P^2$ at four points with a monotone symplectic structure.
Then there exist four Lagrangian spheres and two monotone Lagrangian tori with Maurer Cartan elements 
such that these objects split generate the Fukaya category.
Moreover $\qcoh$ is semi-simple.
\end{theorem}
\begin{remark}
Semi-simplicity of the quantum cohomology has already been known $($\cite{baysim}, \cite{craqua}$)$.
But our proof does not use explicit computation of Gromov-Witten invariants.
\end{remark}
\subsection{Plan of the paper}
In \S \ref{secainf}, we recall some definitions on $\ainf$ categories and Hochschild invariants.
Moreover, we show a split generation criterion for cyclic $\ainf$ categories.
In \S \ref{sec3}, we review Lagrangian intersection Floer theory.
In \S \ref{sec4}, we give some assumptions on Fukaya categories and open-closed maps.
Using these assumptions, we show the main result.
In \S \ref{sec5}, we construct field factors of the quantum cohomology by using Floer theory of Lagrangian submanifolds.
In \S \ref{ex}, we compute the Fukaya category and the quantum cohomology of a blow-up of $\C P^2$ at four points.
\subsection{Acknowledgements}
The author would like to express his gratitude to his supervisor Professor Shinobu Hosono. 
Some parts of this paper is a extended version of the author's Ph.D. thesis. 
The author would like to thank Yuuki Shiraishi for explaining to the author the importance of Shklyarov's pairing, 
Sheel Ganatra for informing the author that he also has obtained a version of the main result after the authors talk at Kyoto(\cite{talk}).
The author also would like to thank Professor Hiroshi Ohta, Professor Kaoru Ono, Daiske Inoue, Tatsuki Kuwagaki, Makoto Miura for helpful discussions.
This work was supported by JSPS KAKENHI Grant Number 12J09640,
JSPS  Grant-in-Aid for Scientific Research number 23224002, 
and JSPS  Grant-in-Aid for Scientific Research number15H02054.


\section{Preliminaries on $\ainf$ categories}\label{secainf}
In \S \ref{secainf},  we recall some definitions about $\ainf$ categories and Hochschild invariants.
We mainly follow \cite{shefor}.
Throughout this section, $\ke$ denotes a field of characteristic zero.
Complexes (chain or cochain) mean
$2$-periodic complexes, i.e., $\Z/2\Z$-graded vector spaces with odd degree morphisms $d$ such that 
$d^2=0.$
For a homogeneous element $x$ of a $\Z/2\Z$-graded vector space, 
$|x|$ denotes the degree of $x$.
Set $|x|':=|x|-1$.
\subsection{$\ainf$ categories}\label{subainf}
We first recall the definition of $\ainf$ categories.
\begin{definition}\
\begin{enumerate}
\item A curved $\ainf$ category $\A$ consists of the following data:
\begin{itemize}
\item A set of objects $\ob \A$.
\item $\Z/2\Z$-graded $\ke$-vector spaces $\Hom{\A}{X}{Y}$ for $X,Y \in \ob \A$.
\item $\ke$-linear maps 
    \[\enm_s :\Hom{\A}{X_0}{X_1} \otimes \Hom{\A}{X_1}{X_2} \otimes 
       \cdots \otimes \Hom{\A}{X_{s-1}}{X_s} \to \Hom{\A}{X_0}{X_s}\]
of degree $2-s \ (0 \le s)$ which satisfy the $\ainf$-relations, i.e., 
\begin{align*}
  \sum_{\substack{0 \leq i,  j \leq s \\ i+j \leq s }} (-1)^{|x_1|'+\cdots +|x_i|'}
  \enm_{s-j+1}(x_1, \dots, x_i, \enm_j(x_{i+1},\dots, x_{i+j}),x_{i+j+1}, \dots, x_s)=0,
\end{align*}
here $x_i$ are homogeneous elements of $\Hom{\A}{X_{i-1}}{X_i}$. 
\end{itemize}

\item A curved $\ainf$ category $\A$ is called finite if $\Hom{\A}{X}{Y}$ is a finite dimensional
         $\ke$-vector space for each $X,Y$. 

\item A curved $\ainf$ category $\A$ is called $($strictly$)$ unital if for each $X \in \ob\A$,
         there exists an even element $1_X \in \Hom{\A}{X}{X}$ called a unit of $X$ such that
\begin{align*}\label{unit}
\enm_s(x_1,\dots 1_X,\dots x_{s-1})=0 \ \text{for} \ s \neq 2,\  
\enm_2(1_X, x)=(-1)^{|x|}\enm_2(x, 1_X)=x.
\end{align*}

\item An $\ainf$ category is a curved $\ainf$ category with $\enm_0=0$.
\end{enumerate}
\end{definition}
Let $\A$ be an $\ainf$ category.
Then, by the $\ainf$-relations, we see that $\enm_1\circ \enm_1=0$ and $\enm_2$ naturally induce maps  
\[ [\enm_2] : H^\bullet\Hom{\A}{X_0}{X_1} \otimes H^\bullet\Hom{\A}{X_1}{X_2}
    \to H^\bullet\Hom{A}{X_0}{X_2},\]
where the cohomology of $\Hom{A}{X}{Y}$ with respect to $\enm_1$ is denoted by $H^\bullet\Hom{A}{X}{Y}$.
We define the cohomology $\ainf$ category $H^\bullet(\A)_\infty$ of $\A$ by the following:
\begin{itemize}
\item $\ob H^\bullet(\A)_\infty:=\ob \A$.
\item $\Hom{{H^\bullet(\A)_\infty}}{X}{Y}:=H^\bullet \Hom{\A}{X}{Y}$.
\item The $\ainf$ structure maps $m_s$ of $H^\bullet(\A)_\infty$ are zero for $s \neq2$ and $[\enm_2]$ for $s=2$.
\end{itemize}
We can easily see that these data define an $\ainf$ category.
If $\A$ is unital, then we can also define the cohomology category $H^\bullet(\A)$.
The objects are the same as $H^\bullet(\A)_\infty$ and 
morphisms are defined by 
\[\Hom{{H^\bullet(\A)}}{X}{Y}:=\Hom{{H^\bullet(\A)_\infty}}{Y}{X}.\]
The unit $1_X$ defines the identity morphism of $X$, which is also denoted by $1_X.$
The composition of homogeneous morphisms
$g \in \Hom{{H^\bullet(\A)}}{X}{Y}$ and $f \in \Hom{{H^\bullet(\A)}}{Y}{Z}$ is defined by 
\begin{align}\label{ring}
f \circ g :=(-1)^{|f||g|+|f|}[\enm_2](f, g).
\end{align}
The product $f \cdot g:=f \circ g$ and the unit $1_X$
gives a ring structure on $H^\bullet \Hom{\A}{X}{X}$ 
\begin{remark}
The definition of the category $H^\bullet(\A)$ is different from \cite{shefor}.
\end{remark}

\subsection{Hochschild cohomology}
We recall the definitions of Hochschild cohomology and its ring structure. 
Let $\A$ be a unital $\ainf$ category.
We define the Hochschild cochains of degree $k\in\Z/2\Z$ and length $s$ by
\[CC^k(\A)^s:=\prod_{X_0, X_1, \dots, X_s} 
   \mathrm{Hom}_{\ke}^{k-s}\big( \Hom{\A}{X_0}{X_ 1}\otimes \cdots \otimes \Hom{\A}{X_{s-1}}{X_s}, 
   \ \Hom{\A}{X_0}{X_s} \big), \]
here $\mathrm{Hom}_{\ke}^{k-s}$ denotes the space of $\ke$-linear morphisms of degree $k-s.$ 
Set \[CC^\bullet(\A):=\prod_{s \ge 0}CC^\bullet(\A)^s.\]
For $\varphi \in CC^\bullet(\A)$, the length $s$ part of $\varphi$ is denoted by $\varphi_s$ and
$\Hom{\A}{X}{X}$ part of $\varphi_0$ is denoted by $\varphi_X$.
The Hochschild differential $M^1$ is defined as follows;
\begin{align}
   M^1(\varphi)(x_1, \dots, x_s):=
     &\sum (-1)^{|\varphi|'(|x_1|'+\cdots +|x_i|')}
      \enm_{s-j+1}(x_1,\dots, \varphi_j(x_{i+1}, \dots ),x_{i+j+1},\dots ,x_s) \\
     +&\sum (-1)^{|\varphi|+|x_1|'+\cdots +|x_i|'} \varphi_{s-j+1}
        (x_1,\dots, \enm_j(x_{i+1}, \dots ),x_{i+j+1},\dots ,x_s) \notag
\end{align}
The complex $\big( CC^\bullet(\A), M^1 \big)$ is called the Hochschild cochain complex and its cohomology is called the Hochscild cohomology.
The Hochschild cohomology of $\A$ is denoted by $HH^\bullet(\A)$.

For $\varphi, \psi \in CC^\bullet(\A),$ the product $\varphi \cup \psi$ is defined as follows;  
\begin{align}\label{prod}
\varphi \cup \psi&:=(-1)^{|\varphi||\psi|+|\varphi|} M^2 (\varphi, \psi), \\ 
 M^2(\varphi, \psi)(x_1,\dots, x_s)&:=
   \sum(-1)^\#
   \enm_*(x_1, \dots, \varphi_*(x_{i+1},\dots),\dots, \psi_*(x_{j+1},\dots),\dots, x_s),
\end{align}
here $\#=|\varphi|'(|x_1|'+\cdots +|x_i|')+|\psi|'(|x_1|'+\cdots +|x_j|').$
We define $\uniho \in CC^0(\A)^0$ by
\[(\uniho)_X:= 1_X.\]
These data makes $HH^\bullet(\A)$ into a $\Z/2\Z$-graded ring with the unit $\uniho$.

The graded center $Z^\bullet(\A)$ is defined by the set of length $0$ cocycles of $CC^\bullet(\A).$

Let $\varphi$ be a cocycle of $CC^\bullet(\A)$.
Then $M^1(\varphi)_X=0$ implies $\enm_1(\varphi_X)=0$.
$[\varphi_X]$ denotes the cohomology class of $\varphi_X.$
We easily see that $[\varphi_X]$ only depends on the cohomology class $[\varphi].$  
Moreover, $M^1(\varphi)_1=0$ implies
\begin{align}\label{commute}
[\varphi_X] \cdot [x] = (-1)^{|\varphi|\cdot |x|} [x] \cdot [\varphi_Y]
\end{align}
for $[x] \in H^\bullet \Hom{\A}{X}{Y}.$

Let $\B \subset \A,$ i.e., $\ob\B \subset \ob\A$ and $\B$ is equipped with the natural unital $\ainf$ structure induced from $\A.$
By restriction,  we have a natural morphism $r_\B : CC^\bullet(\A) \to CC^\bullet(\B)$,  which induces a morphism on cohomology \[[r_\B]: HH^\bullet(\A) \to HH^\bullet(\B).\]
Similarly, for $X \in \A,$ we have a morphism $r_X : CC^\bullet(\A) \to \Hom{\A}{X}{X},$ 
which induces a morphism on cohomology \[ [r_X] : HH^\bullet(\A)  \to H^\bullet \Hom{\A}{X}{X}.\]
We note that $[r_\B]$ and $[r_X]$ are ring homomorphisms.

The following lemma is used in \S \ref{decomp}.
\begin{lemma}\label{2.2}
Let $[\varphi] \in HH^\bullet(\A)$ be an idempotent such that $[\varphi_X]=0$ for all $X$. 
Then, for each $N \in \Z_{\ge 0}$, there exists $\varphi^N \in CC^\bullet(\A) $ such that
$M^1(\varphi^N)=0, \ [\varphi^N]=[\varphi]$, and $\varphi^N_s=0$ for $s \le N$. 
\end{lemma}
\begin{proof}
By assumption, there exists $\psi_X$ such that $\varphi_X=\enm_1 (\psi_X)$ for each $X$. 
Set 
\[\psi(x_1, \dots, x_s)=
  \begin{cases}
    0 &1 \le s \\
    \prod \psi_X & s=0
\end{cases},\]
then $(\varphi-M^1 \psi)_X=0$.
Hence we can assume $\varphi_X=0$, which implies the $N=0$ case of the lemma.
Assume that the lemma holds for $N$.
Then $\varphi^N=\varphi^N \cup \varphi^N+M^1(\psi)$ for some $\psi$.
By the definition of $\cup$, we see that $(\varphi^N \cup \varphi^N)(x_1, \dots, x_s)=0$ for $s \le N+1$.
Hence $\varphi^N- M^1(\psi)$ satisfies the conditions of $\varphi^{N+1}$, which implies the lemma by induction.
\end{proof}

\subsection{Hochschild homology}
Let $\A$ be a unital $\ainf$ category.
We define the Hochschild chains 
as follows:
\[CC_\bullet(\A):=\bigoplus_{X_0, X_1, \dots X_s}\Hom{\A}{X_0}{X_1} \otimes \Hom{\A}{X_1}{X_2} \otimes 
\cdots \otimes \Hom{\A}{X_{s-1}}{X_s} \otimes \Hom{\A}{X_s}{X_0}\]
An element
\[x_0 \otimes x_1 \otimes \cdots \otimes x_s \in \Hom{\A}{X_0}{X_1} \otimes \Hom{\A}{X_1}{X_2} \otimes \cdots \otimes \Hom{\A}{X_{s-1}}{X_s} \otimes \Hom{\A}{X_s}{X_0}\]
 is also denoted by
$x_0[x_1|x_2|\cdots|x_s]$ or $\X$ if it is considered as an element of $CC_\bullet(\A)$.
The degree of $x_0[x_1|x_2|\cdots|x_s]$ is defined by $|x_0|+|x_1|'+\cdots +|x_s|'$.
Set ${|\X|'}_i^j:=|x_i|'+|x_{i+1}|'+\cdots +|x_j|'$ for $i \le j$.
The differential $b$ is defined by
\begin{align}
b(x_0[x_1|\cdots|x_s]):=&\sum (-1)^{{|\X|'}_0^i}x_0[\cdots|\enm_*(x_{i+1},\dots)|\cdots|x_s]\\
   +&\sum (-1)^{{|\X|'}_0^i \cdot {|\X|'}_{i+1}^s}\enm_*(x_{i+1},\dots, x_0, \dots x_j)[x_{j+1}|\cdots|x_i]. \notag
\end{align}
Then $\big(CC_\bullet(\A), b \big)$ is a chain complex and its cohomology $HH_\bullet(\A)$ is called the Hochschild homology. 

We define $b^{1|1}:CC^\bullet(\A) \otimes CC_\bullet(\A) \to CC_\bullet(\A)$ by 
\begin{align}\label{mod}
  b^{1|1}(\varphi; \X):=\sum(-1)^\# 
    \enm_*(x_{i+1}, \dots, \varphi_*(x_{j+1},\dots), \dots, x_0, \dots, x_k)[x_{k+1}|\cdots|x_i],
 \end{align}
here $\#={|\X|'}_0^i \cdot {|\X|'}_{i+1}^s+|\varphi|' \cdot {|\X|'}_{i+1}^j$.
We also define 
\[b^{2|1}:CC^\bullet(\A)^{\otimes2} \otimes CC_\bullet(\A) \to CC_\bullet(\A)[-1]\]
by
\[ b^{2|1}(\varphi, \psi; \X):=\sum(-1)^\# 
    \enm_*(x_{i+1}, \dots, \varphi_*(x_{j+1},\dots), \dots, \psi_*(x_{k+1} \dots), \dots, x_0, \dots, x_l)[x_{l+1}|\cdots|x_i],
     \]
here $\#={|\X|'}_0^i \cdot {|\X|'}_{i+1}^s+|\varphi|' \cdot {|\X|'}_{i+1}^j+|\psi|' \cdot {|\X|'}_{i+1}^k.$
Then the operators $M^1, M^2, b, b^{1|1}, b^{2|1}$ satisfy the following relations (\cite[Theorem 1.9]{getcar}, see also \cite{shefor}):
\begin{align}\label{modrel}
&b(b^{1|1}(\varphi;\X))+(-1)^{|\varphi|'}b^{1|1}(\varphi;b(\X))+b^{1|1}(M^1(\varphi);\X)=0, \\ \label{modrel2}
&b(b^{2|1}(\varphi, \psi; \X))+(-1)^{|\varphi|'}b^{1|1}(\varphi; b^{1|1}(\psi; \X))+
  (-1)^{|\varphi|'+|\psi|'}b^{2|1}(\varphi, \psi; b(\X))\\
+&b^{1|1}(M^2(\varphi, \psi);  \X)+b^{2|1}(M^1(\varphi), \psi; \X)
  +(-1)^{|\varphi|'}b^{2|1}(\varphi, M^1(\psi); \X)=0.\notag
\end{align}
Set $\varphi \cap \X:=(-1)^{|\varphi| \cdot |\X|+|\varphi|}b^{1|1}(\varphi;\X)$.
Then $\cap$ induces an operator of cohomology
\[\cap:HH^\bullet(\A) \otimes HH_\bullet(\A) \to HH_\bullet(\A),\]
which makes $HH_\bullet(\A)$ into a $\Z/2\Z$-graded left $HH^\bullet(\A)$-module.
\begin{remark}
The definitions of the operators $M^1, M^2, b, b^{1|1}, b^{2|1}$ are the same as \cite{shefor}.
But $\cup$ and $\cap$ are different from \cite{shefor}.
\end{remark}
Let $\B \subset \A$ be a subcategory.
 $c_\B$ denotes the inclusion $CC_\bullet(\B) \subset CC_\bullet(\A)$, 
which induces a morphism on cohomology 
\[ [c_\B] : HH_\bullet(\B) \to HH_\bullet(\A).\]
Similarly, for $X \in \A,$ the inclusion $\Hom{\A}{X}{X} \subset CC_\bullet(\A)$ is denoted by $c_X.$
The morphism $c_X$ induces a morphism on cohomology
\[[c_X]: H^\bullet \Hom{\A}{X}{X}  \to HH_\bullet(\A).\]
By the definition of $b^{1|1}$, we easily see that 
\begin{align}\label{dual}
b^{1|1}(\varphi ; c_\B(\X))=c_\B \circ b^{1|1}(r_\B(\varphi) ; \X),
\end{align}
here $\varphi \in CC^\bullet(\A)$ and $\X \in CC_\bullet(\B).$

\subsection{The Mukai pairing}\label{mukai}
In \S \ref{mukai}, we consider a finite unital $\ainf$ category $\A$.
We recall the definition of the Mukai pairing on $HH_\bullet(\A)$ (\cite{shkhir} for dg cases, \cite[Proposition 5.22]{shefor} for $\ainf$ cases).

For $\X=x_0[x_1|\cdots|x_s], \ \X'=x'_0[x'_1|\cdots|x'_t] \in CC_\bullet(\A),$ we define $\muk{\X}{\X'}$ by 
\begin{align*}
\sum \tr(x'' \to (-1)^\#
   \enm_*(x_{i+1}, \dots, x_0, \dots, x_j, \enm_*(x_{j+1},\dots, x_i, x'', x'_{k+1}, \dots, x'_0, \dots, x'_l),
              x'_{l+1}, \dots, x'_k),
\end{align*}
here $\#=1+{|\X|'}_{i+1}^s+{|\X|'}_0^j+|x''| \cdot |\X'|+
             {|\X|'}_0^i \cdot {|\X|'}_{i+1}^s+{|\X'|'}_0^k \cdot {|\X'|'}_{k+1}^t$.
This pairing satisfy \[\muk{b(\X)}{\X'}+(-1)^{|\X|}\muk{\X}{b(\X')}=0\]
and induces a pairing $\muk{}{}$ on $HH_\bullet(\A)$ of degree zero.
For example, if $\X=x_0, \X'=x_0' \in HH_\bullet(\A)$, then we have 
\begin{align}\label{mukfor}
\muk{\X}{\X'}=\str\big{(}x'' \to (-1)^{|x_0|\cdot|x''|+|x_0| \cdot |x_0'|}x_0 \circ x'' \circ x_0'\big{)},
\end{align}
where $\str$ is the supertrace of the $\Z/2\Z$-graded vector space $H^\bullet \Hom{\A}{X_0}{X_0'}$.
By definition, for $\B \subset \A$, we see that $c_\B$ preserves the Mukai pairing.
We recall the following theorem:
\begin{theorem}[\cite{shkhir} for dg cases, {\cite[Proposition 5.24]{shefor}} for $\ainf$ cases]\label{shk}
Let $\A$ be a finite unital $\ainf$ category.
Suppose that $\A$ is smooth $($see, e.g., \cite{konnot} for the definition$)$.
Then $\muk{}{}$ is a non-degenerate pairing on $HH_\bullet(\A)$.
\end{theorem}
\begin{remark}
For a smooth finite unital $\ainf$ category, the Hochschild homology is finite dimensional.
Hence $\muk{}{}$ is perfect.
\end{remark}
The next lemma is used in \S \ref{decomp}.
\begin{lemma}\label{2.4}
Let $\A$ be a finite unital $\ainf$ category and $\A_1, \A_2 \subset \A.$
Suppose that \[H^\bullet \Hom{\A}{X_1}{X_2}=0\] for $X_i \in \A_i.$
Let $\X_i \in HH_\bullet(\A_i).$ 
Then $\muk{\X_1}{\X_2}=0.$
Note that $\X_i$ are naturally considered as elements of $HH_\bullet(\A).$
\end{lemma}
\begin{proof}
By taking a minimal model of $\A$, we can assume that $\Hom{\A}{X_1}{X_2}=0$ for $X_i \in \A_i.$
Then the lemma easily follows from the definition of $\muk{}{}.$  
\end{proof}


\subsection{Cyclic symmetry}\label{cyclic symmetry}
In \S \ref{cyclic symmetry}, we refer the reader to \cite{fukcyc}.
\begin{definition}
Let $\A$ be a unital finite curved $\ainf$ category and $n$ be a natural number.
$\A$ is called $n$-cyclic if, for each $X, Y$, there exists  a non-degenerate pairing 
\[\cyc{}{} : \Hom{\A}{X}{Y} \otimes \Hom{\A}{Y}{X} \to \ke \]
of degree $n$ which satisfy the following relations:
\begin{align}\label{cyc}
\cyc{\enm_{s-1}(x_1, \dots, x_{s-1})}{x_s}=
   (-1)^{|x_s|'(|x_1|'+\cdots +|x_{s-1}|')} \cyc{\enm_{s-1}(x_s, x_1, \dots, x_{s-2})}{x_{s-1}}.
\end{align}
\end{definition}
Combined with unitality, we have 
\begin{align}\label{symm}
\cyc{x_1}{x_2}&=(-1)^{1+|x_1|'\cdot |x_2|'} \cyc{x_2}{x_1}.
\end{align}
By Equations (\ref{cyc}) and (\ref{symm}), we easily see that
\begin{align}\label{symmm}
\cyc{\enm_{s-1}(x_1, x_2, \dots, x_{s-1})}{x_s}=(-1)^{|x_1|}\cyc{x_1}{\enm_{s-1}(x_2, x_3, \dots, x_s)}
\end{align}
Especially, we have $\cyc{\enm_1(x_1)}{x_2}=(-1)^{|x_1|}\cyc{x_1}{\enm_1(x_2)}.$
Hence we see that $\cyc{}{}$ induces a non-degenerate pairing
\[H^\bullet \Hom{\A}{X}{Y}\otimes H^\bullet \Hom{\A}{Y}{X} \to \ke,\]
which is also denoted by $\cyc{}{}.$
Therefore $H^\bullet(\A)_\infty$ is also an $n$-cyclic $\ainf$ category.
We define a morphism $\int_\A : CC_\bullet(\A) \to \ke$ of degree $-n$ by 
\begin{align}\label{tr}
\int_\A (x_0[x_1|\cdots|x_s]):=
 \begin{cases}
   \cyc{1_{X_0}}{x_0} &\text{if} \ s=0\\
   0 &\text{if}\ s \neq 0
  \end{cases}.
\end{align}
We call $\int_\A$ a trace map.
We note that $\int_\A \enm_2(x_0, x_1)=\cyc{x_0}{x_1}$.
Using this relation, we easily see that
\begin{align}\label{trrel}
\int_\A b(x_0[x_1|\cdots|x_s])=0.
\end{align}
Hence $\int_\A$ induces a morphism on cohomology, which is also denoted by $\int_\A$.
For $\varphi \in CC^\bullet (\A)$ and $\X=x_0[x_1|\cdots|x_s] \in CC_\bullet(\A),$
we define $\cyc{\varphi}{\X}$ by 
\begin{align}
\cyc{\varphi}{\X}:=\int_\A b^{1|1}(\varphi; \X).
\end{align}
We note that if $\varphi, \X \in \Hom{\A}{X}{X}$ then $\cyc{\varphi}{\X}$
is equal to the pairing on $\Hom{\A}{X}{X}$. 
By straight forward calculation, we see that 
\[\cyc{\varphi}{\X}=(-1)^{|x_0|' \cdot {|\X|'}_1^s} \cyc{\varphi(x_1, \dots, x_s)}{x_0}.\]
Hence $\cyc{}{}$ induces an isomorphism $CC^\bullet(\A) \cong CC_\bullet(\A)[n]^\vee.$
Moreover $\cyc{}{}$ satisfies
\[\cyc{M^1(\varphi)}{\X}+(-1)^{|\varphi|'}\cyc{\varphi}{b(\X)}=0\]
by Equations (\ref{modrel}) and (\ref{trrel}).
Hence $\cyc{}{}$ induces a  pairing of degree $n$
\[HH^\bullet(\A)\otimes HH_\bullet(\A) \to \ke,\]
which is also denoted by $\cyc{}{}$.
Moreover $\cyc{}{}$ induces an isomorphism $HH^\bullet(\A) \cong HH_\bullet(\A)[n]^\vee.$
By Equation (\ref{modrel2}), we easily see that 
\begin{align}\label{modcyc}
\cyc{\varphi \cup \psi}{\X}=(-1)^{n\cdot |\psi|}\cyc{\varphi}{\psi \cap \X}.
\end{align}

We define $Z: CC_\bullet(\A) \to CC^\bullet(\A)[n]$ by the formula
\[\muk{\X}{\X'}=\cyc{Z(\X)}{\X'},\]
which induces a morphism $[Z]: HH_\bullet(\A) \to HH^\bullet(\A)[n]$.
Set $Z_X:= r_X \circ Z$ and $[Z_X]:= [r_X] \circ [Z].$
By Theorem \ref{shk}, $[Z]$ is an isomorphism if $\A$ is smooth.

\subsection{Split generation}
Let $\A$ be an $n$-cyclic $\ainf$ category and $\B \subset \A$ be a subcategory.
For each $X, Y \in \A$, take a homogeneous basis $\{e_\alpha^{X,Y} \}$ of $\Hom{\A}{X}{Y}$.
We define $\{e^\alpha_{Y, X} \} \subset \Hom{\A}{Y}{X}$ by the formula
$\cyc{e^\alpha_{Y, X}}{e_\beta^{X,Y}}=\delta^\alpha_\beta$.
For an element $x \in \Hom{\A}{X}{X} \subset CC_\bullet(\A)$, we have 
$\cyc{Z_X(\X)}{x}=\cyc{Z(\X)}{x}=\muk{\X}{x}.$
Using Equation (\ref{symmm}), we see
\begin{align*}
\muk{\X}{&x}\\
                &= \sum \tr\big(x' \to (-1)^\#
                     \enm_*(x_{i+1}, \dots, x_0, \dots, x_j, \enm_*(x_{j+1},\dots, x_i, x', x)\big)\\
                &= \sum (-1)^\#
                     \cyc{e^\alpha_{X, X_{i+1}}}
                      {\enm_*(x_{i+1}, \dots, x_0, \dots, x_j, \enm_*(x_{j+1},\dots, x_i,e_\alpha^{X_{i+1}, X} , x))}\\
                &=\sum (-1)^{\#+|e^\alpha_{X, X_{i+1}}|}
                    \cyc{\enm_*(e^\alpha_{X, X_{i+1}}, x_{i+1}, \dots, x_0, \dots, x_j)}
                    {\enm_*(x_{j+1},\dots, x_i,e_\alpha^{X_{i+1}, X} , x))}\\
                &=\sum (-1)^{\#+|e^\alpha_{X, X_{i+1}}|+\#'}
                    \cyc{\enm_*(\enm_*(e^\alpha_{X, X_{i+1}}, x_{i+1}, \dots, x_0, \dots, x_j),
                     x_{j+1},\dots, x_i, e_\alpha^{X_{i+1}, X})}{x},\\
\end{align*}
here $\#=1+{|\X|'}_{i+1}^s+{|\X|'}_{0}^j+|e_\alpha^{X_{i+1}, X}| \cdot |x|+{|\X|'}_0^i \cdot {|\X|'}_{i+1}^s$ and
       $\#'=|e^\alpha_{X, X_{i+1}}|'+{|\X|'}_{i+1}^s+{|\X|'}_{0}^j.$
Thus we have
\begin{align}
Z_X(\X)=\sum(-1)^\#
\enm_*(\enm_*(e^\alpha_{X, X_{i+1}}, x_{i+1}, \dots, x_0, \dots, x_j), x_{j+1}, \dots, x_i, e_\alpha^{X_{i+1}, X}),
\end{align}
here $\#=|e_\alpha^{X_{i+1}, X}| \cdot |\X|+{|\X|'}_0^i \cdot {|\X|'}_{i+1}^s$. 

For $K \in \A$, we define a chain complex $\Y^r_K \otimes_\B \Y^l_K$ as follows:
\[\Y^r_K \otimes_\B \Y^l_K:=
\!\bigoplus_{\substack{X_0, \dots, X_s \in \B \\ 0 \le s}}\! \Hom{\A}{K}{X_0} \otimes \Hom{\A}{X_0}{X_1}[1] \otimes \cdots \otimes
\Hom{\A}{X_{s-1}}{X_s }[1] \otimes \Hom{\A}{X_s}{K}\]
with the differential 
\begin{align*}
d(m\otimes x_1\otimes \cdots \otimes x_s \otimes n) :=-&\sum \enm_*(m, x_1, \dots, x_i) \otimes x_{i+1} \otimes \cdots \otimes x_s \otimes n \\
+&\sum(-1)^{|m|+|x_1|'+\cdots +|x_i|'}m \otimes \cdots \otimes \enm_*(x_{i+1}, \dots, x_j) \otimes \cdots \otimes n  \notag \\
+&\sum(-1)^{|m|+|x_1|'+\cdots +|x_i|'}m \otimes \cdots \otimes \enm_*(x_{i+1}, \dots, x_s, n) \notag
\end{align*}
We define $m^K: \Y^r_K \otimes_\B \Y^l_K \to \Hom{\A}{K}{K}$ by 
\[m^K(m\otimes x_1 \otimes \cdots \otimes x_s \otimes n):=\enm_{s+2}(m, x_1, \dots, x_s, n).\]
By the $\ainf$ relations, we see $m^K \circ d= \enm_1 \circ m^K$. The morphism between cohomology induced by $m^K$ is denoted by $[m^K]$.
We recall the following result which is due to Abouzaid \cite{abogeo}.

\begin{theorem}\label{sp}
If $1_K \in \im([m^K])$, then $K$ is split generated by $\B$, i.e.,
$K \in D^{\pi}(\B) \subset D^{\pi}(\A)$.
\end{theorem}

\begin{remark}
We briefly explain the definition of the $($split-closed\ \!$)$ derived category $D^\pi(\A)$ of a unital $\ainf$ category $\A$.
Let $\mathrm{Ch}(\ke)$ be the dg category of chain complexes over $\ke$.
The $\ainf$ category of one-sided twisted complexes over $\A \otimes_{\ke} \mathrm{Ch}(\ke)$ is denoted by $\mathrm{Tw}(\A)\  (see, $e.g., \cite[Section 8]{aspdir}$)$ and its split closure is denoted by $\Pi \mathrm{Tw} (\A)\  (see, $e.g., \cite[Section 4]{seibook}$)$.
Then $D^\pi(\A)$ is defined by $H^0(\Pi \mathrm{Tw}(\A))$.
For $\B \subset \A,$ we say that $\B$ split generates $\A$ if the inclusion $D^\pi(\B) \subset D^\pi(\A)$ gives an equivalence. 
\end{remark}

We define $CC_\bullet(\Delta): CC_\bullet(\B) \to \Y^r_K \otimes_\B \Y^l_K $ by 
\begin{align*}
CC_\bullet(\Delta)(\X):=\sum (-1)^\# \enm_*(e^\alpha_{K, X_{i+1}}, x_{i+1}, \dots, x_0, \dots, x_j)
                                 \otimes x_{j+1} \otimes \cdots \otimes x_i \otimes e_\alpha^{X_{i+1}, K},
\end{align*}
here $\#=|e_\alpha^{X_{i+1}, K}| \cdot |\X|+{|\X|'}_0^i \cdot {|\X|'}_{i+1}^s$.

\begin{lemma}
$d \circ CC_\bullet(\Delta)+(-1)^n CC_\bullet(\Delta) \circ b=0$.
\end{lemma}
\begin{proof}
By the $\ainf$ relations, we have 
\begin{align*}
(d \circ CC&_\bullet(\Delta)+(-1)^n CC_\bullet(\Delta) \circ b)(\X)=\\
&\sum (-1)^\# \enm_*(\enm_*(e^\alpha_{K, X_{i+1}}, x_{i+1}, \dots, x_j), x_{j+1}, \dots, x_0, \dots, x_k)
        \otimes x_{k+1} \otimes \cdots \otimes x_i \otimes e_\alpha^{X_{i+1}, K} \\
+&\sum(-1)^{\#'}\enm_*(e^\alpha_{K, X_{i+1}}, x_{i+1}, \dots, x_j) \otimes \cdots \otimes 
 \enm_*(x_{k+1}, \dots, x_i,  e_\alpha^{X_{i+1}, K}),
\end{align*}
here $\#=|e_\alpha^{X_{i+1}, K}| \cdot |\X|+{|\X|'}_0^i \cdot {|\X|'}_{i+1}^s$ and
$\#'=\#+|e^\alpha_{K, X_{i+1}}|'+{|\X|'}_{i+1}^j$.\\
Set \[\enm_*(x_{i+1}, \dots, x_j, e_\alpha^{X_{j+1}, K})=\sum_\beta m_\alpha^\beta e_\beta^{X_{i+1}, K}.\]
Note that $m_\alpha^\beta=0$ if ${|\X|'}_{i+1}^j+|e_\alpha^{X_{j+1},K}|+1 \neq |e_\beta^{X_{i+1}, K}|$.
Then we have 
\begin{align*}
\enm_*(e^\alpha_{K, X_{i+1}}, x_{i+1}, \dots, x_j)&=
\sum_\beta \langle \enm_*(e^\alpha_{K, X_{i+1}}, x_{i+1}, \dots, x_j),
         e_\beta^{X_{j+1}, K}\rangle e^\beta_{K, X_{j+1}}\\
&=\sum_\beta (-1)^{|e^\alpha_{K, X_{i+1}}|} 
   \langle e^\alpha_{K, X_{i+1}}, \enm_*(x_{i+1}, \dots, x_j, e_\beta^{X_{j+1}, K})\rangle e^\beta_{K, X_{j+1}}\\
&=\sum_\beta (-1)^{|e^\beta_{K, X_{i+1}}|+{|\X|'}_{i+1}^j+1} m_\beta^\alpha e^\beta_{K, X_{j+1}}.
\end{align*}
Using this equality, we see that $(d \circ CC_\bullet(\Delta)+(-1)^n CC_\bullet(\Delta) \circ b)(\X)=0$.
\end{proof}
By this lemma, we see that $CC_\bullet (\Delta)$ induces a morphism on cohomology.
By construction, we have $ m^K \circ CC_\bullet(\Delta)=Z_K \circ c_\B$.
Combined with Theorem \ref{sp}, we have the following theorem, which is a version of Abouzaid's generating criterion \cite{abogeo}.
\begin{theorem}\label{spgen}
Let $\A$ be an $n$-cyclic $\ainf$ category and $\B \subset \A$ be a subcategory.
Let $K$ be an object of $\A$. 
Suppose that $1_K \in \im([Z_K] \circ [c_\B])$.
Then $K$ is split generated by $\B$. 
\end{theorem}


\section{Preliminaries on Floer cohomology and Quantum cohomology}\label{sec3}
In \S \ref{sec3}, we review Floer cohomology of Lagrangian submanifolds (\cite{fooo1}).
\subsection{Floer cohomology of Lagrangian submanifolds}\label{lagfloer}
Let $(X, \omega)$ be a compact symplectic manifold of dimension $2n$. 
We fix an $\omega$-compatible almost complex structure on $X$.
Let $L$ be a compact oriented Lagrangian submanifold of $(X, \omega)$ with a spin structure 
and $\rho : \pi_1(L) \to \C^\times$ be a group homomorphism which is considered as a local system on $L$.
Let $\mu_L \in H^2(X, L; \Z)$ be the Maslov class. 
We denote by $\Lambda$ the universal Novikov field over $\C$.
Namely,
\[\Lambda= \left\{\sum_{i=0}^\infty a_iT^{\lambda_i}\middle|a_i\in \mathbb{C} \ \lambda_i\in \mathbb{R} \ \lim_{i\to\infty}\lambda_i=\infty\right\}.\]
This is an algebraically closed valuation field.
The valuation is defined by 
\[ \mathrm{val}(\sum_{i=0}^\infty a_iT^{\lambda_i}):=\min\{\lambda_i \mid a_i \neq 0 \}.\]
We denote by $\Lambda_0$ the valuation ring of $\Lambda$ and by $\Lambda_+$ the maximal ideal of $\Lambda_0$.
Let $\mathscr{M}_{s+1}(L; \beta)$ be the moduli space of stable disks bounded by $L$ of homology class $\beta \in H_2(X, L; \Z)$ with $s+1$ counterclockwise ordered boundary marked points.
 
Fukaya-Oh-Ohta-Ono construct $\C$-linear maps
\[m^\can_{s, \beta}:\coh{L}{\C}^{\otimes s} \to \coh{L}{\C} \]
for $s \in \Z_{\ge0}$ and $\beta \in H_2(X, L; \Z)$ by using $\mathscr{M}_{s+1}(L; \beta)$
(see, e.g., \cite{fukcyc}, \cite{fooo1}).
Let $1_L$ be the unit of $\coh{L}{\C}$ and $\pd{}{}{L}$ be the Poincar\'e pairing, i.e., 
\[ \pd{x_1}{x_2}{L}:= \int_L x_1 \cup x_2.\]
Set \[\cyc{x_1}{x_2}:=(-1)^{|x_1| \cdot |x_2|+|x_1|}\pd{x_1}{x_2}{L}.\]
By construction, $m^\can_{s, \beta}$ satisfy the following conditions:
\begin{condition}[{\cite[Corollary 12.1]{fukcyc}}]\label{floer}\
\begin{enumerate}
\item $m^\can_{s, \beta}$ is degree $2-s-\mu_L(\beta)$
          with respect to the cohomological $\Z$ grading of $\coh{L}{\C}$. \\
\item If $m^\can_{s, \beta} \neq 0$,
         then there exist $\beta_1, \dots, \beta_m$ such that 
         $\beta=\beta_1+\cdots +\beta_m$ and $\mathscr{M}_{s+1}(L; \beta_i) \neq \emptyset$. \\
\item $m^\can_{s, 0}(x_1,\dots, x_s)=
            \begin{cases}
                 0 &\text{if} \ s=0, 1 \\
                 (-1)^{|x_1||x_2|+|x_1|}x_1 \cup x_2 & \text{if}\  s=2
             \end{cases}. $\\
\item $\displaystyle{\sum_{\substack{s_1+s_2=s+1 \\ \beta_1+\beta_2=\beta}}}
          (-1)^{|x_1|'+\cdots+|x_i|'} 
          m^\can_{s_1, \beta_1}(x_1, \dots, m^\can_{s_2, \beta_2}(x_{i+1}, \dots), \dots, x_s)=0.$ \\
\item $m^\can_{s, \beta}(\dots, 1_L,\dots)=0$ if $\beta \neq 0$ or $s \ge 3$. \\
\item $\cyc{m^\can_{s, \beta}(x_1,\dots, x_s)}{x_0}=
          (-1)^{|x_0|' \cdot (|x_1|'+\cdots +|x_s|')} \cyc{m^\can_{s, \beta}(x_0, x_1, \dots, x_{s-1})}{x_s}$. 
\end{enumerate}
\end{condition}
We define
$m^\can_{s, \rho} : \coh{L}{\Lambda}^{\otimes s} \to \coh{L}{\Lambda}$ by
\[m^\can_{s,\rho}:=\displaystyle{\sum_{\beta \in H_2(X, L; \Z)}}
	  \rho(\partial \beta) T^{\frac{\omega(\beta)}{2\pi}} m^\can_{s, \beta}.\]
Here $\coh{L}{\Lambda}$ is considered as a $\Z/2\Z$-graded vector space. 
By the above conditions, we have the following theorem:
\begin{theorem}[see, e.g., {\cite[Section 2.2]{foootoricdeg}}]
$\big(\coh{L}{\Lambda}, m^\can_{s, \rho}, 1_L, \cyc{}{}\big)$
defines an $n$-cyclic curved $\ainf$ algebra, i.e., an $n$-cyclic curved $\ainf$ category with one object. 
\end{theorem}
\subsection{Maurer-Cartan elements}
To construct (non-curved) $\ainf$ algebras, we introduce (weak) Maurer-Cartan elements.
Let $b_+ \in H^\text{odd}(L; \Lambda_+)$ be an odd degree element. 
The pair $(\rho, b_+)$ is denoted by $b$.
Set
\[m^\can_{s, b}(x_1, \dots, x_s):=\!\!
  \sum_{0 \le l_0, l_1, \dots, l_s} \!\! m^\can_{s+l_0+ \cdots +l_s, \rho}(\overbrace{b_+, \dots, b_+}^{l_0}, x_1,
   \overbrace{b_+, \dots, b_+}^{l_1}, x_2, b_+, \dots, b_+, x_s, 
   \overbrace{b_+, \dots, b_+}^{l_s}).\]
\begin{remark}
Since $b_+ \in \coh{L}{\Lambda_+},$ $m^\can_{s, b}$ converge with respect to the energy filtration. 
\end{remark}
We easily see that operators $m^\can_{s, b}$ also satisfy the $\ainf$ relations and the curvature
$m^\can_{0, b}(1)$ is equal to $\displaystyle{\sum_{0 \le s}} m^\can_{s, \rho}(b_+, \dots, b_+)$.
\begin{definition}
$b=(\rho, b_+)$ is called a (weak) Maurer-Cartan element if $b$ satisfies the Maurer-Cartan equation, i.e., 
there exists $W_{L, b} \in \Lambda$ such that 
\begin{align*}
\sum_{0 \le s} m^\can_{s, \rho}(b_+, \dots, b_+)=W_{L,b} \cdot 1_L.
\end{align*}
\end{definition}
For a Maurer-Cartan element $b$, operators $m^\can_{s, b} \ (s \ge 1)$ define an $\ainf$ algebra.
Moreover, using Condition \ref{floer},  we easily see the following: 
\begin{theorem}\label{oneobj}
Let $b=(\rho, b_+)$ be a Maurer-Cartan element. 
Then $\big(\coh{L}{\Lambda}, m^\can_{s, b}, 1_L, \cyc{}{}\big)$
defines an $n$-cyclic $\ainf$ algebra, here we set $m^\can_{0, b}=0$.
\end{theorem}
The cohomology of $\coh{L}{\Lambda}$ with respect to $m^\can_{1, b}$ is denoted by $HF^\bullet(L, b)$. 
This is naturally considered as a $\Z/2\Z$-graded ring with the unit $1_L$. 

\begin{definition}
Let $L$ be a compact oriented Lagrangian submanifold with a spin structure.
$L$ is called weakly unobstructed if there exists a Maurer-Cartan element.
\end{definition}
We will use the next lemma to state the divisor axiom \big(Assumption \ref{oc} (2)\big).
\begin{lemma}
Let $\eta \in H^2_c(X \setminus L; \C)$ be a compact support cohomology class and $b=(\rho, b_+)$ be a Maurer-Cartan element.
Set 
\begin{align}\label{defdiv}
i^*_{L,b}(\eta):=\sum_{\substack{\beta \in H^2(X, L; \Z) \\ 0 \le s}}\pd{\eta}{\beta}{} \rho(\partial \beta)
 T^{\frac{\omega(\beta)}{2\pi}}
                     m^\can_{s, \beta}(b_+, \dots, b_+)
\end{align}
Then $m^\can_{1, b}\big(i^*_{L,b}(\eta)\big)=0$.
\end{lemma}
\begin{proof}
By Condition \ref{floer} (4), we have
\[ m^\can_{1, b}\big(i^*_{L,b}(\eta)\big)+
    \sum\pd{\eta}{\beta}{} \rho(\partial \beta) T^{\frac{\omega(\beta)}{2\pi}}
            m^\can_{s, \beta}(b_+, \dots,b_+, W_{L,b}\cdot 1_L, b_+, \dots, b_+)=0\]
Combined with Condition \ref{floer} (5), we see  $m^\can_{1, b}\big(i^*_{L,b}(\eta)\big)=0.$
\end{proof}
\subsection{Fukaya categories with one Lagrangian submanifold}
For a compact oriented Lagrangian submanifold with a spin structure $L$ and a group homomorphism $\rho$, we define an $n$-cyclic $\ainf$ category $\fuk{L,\rho}$.
$\ob(\fuk{L, \rho})$ is the set of Maurer-Cartan elements $(\rho, b_+)$ of $L$. 
For objects $b_1=(\rho, b_{1, +})$ and $b_2=(\rho, b_{2, +})$, the morphism space is defined by 
\[\Hom{}{b_1}{b_2}:=
 \begin{cases} \coh{L}{\Lambda} & \text{if} \ W_{L, b_1}=W_{L, b_2} \\
                     0  &\text{otherwise}
  \end{cases}.\]
For Maurer-Cartan elements $b_i=(\rho, b_{i, +})$ and $x_i \in \Hom{}{b_{i-1}}{b_i}\  (1 \le i \le s)$, the $\ainf$ structure maps $m^{L, \rho}_s$ are defined by
\[m^{L, \rho}_s(x_1, \dots, x_s):=\!
  \sum_{0 \le l_0, l_1, \dots, l_s} \! m^\can_{s+l_0+ \cdots +l_s, \rho}
   (\overbrace{b_{0, +}, \dots, b_{0, +}}^{l_0}, x_1,b_{1, +}, \dots, b_{s-1,+}, x_s, 
   \overbrace{b_{s, +}, \dots, b_{s, +}}^{l_s})\]
if $W_{L,b_0}=\cdots =W_{L, b_s}$. 
If otherwise, then $m^{L, \rho}_s(x_1, x_2, \dots, x_s)$ is defined to be $0$.
Then operators $m^{L, \rho}_s \ (1 \le s)$ satisfy the $\ainf$ relations
(see, e.g., \cite[Proposition 1.20]{fukflo2}). 
Moreover we easily see the following:
\begin{theorem}\label{onelag}
$\fuk{L, \rho}$ is an $n$-cyclic $\ainf$ category, here units and pairings are the same as Theorem \ref{oneobj}. 
\end{theorem}

\subsection{Quantum cohomology}\label{quantum}
In \S \ref{quantum}, we introduce some notations about quantum cohomology.
See \cite{fukarn} for more details. 

Let \[\gw_l(\alpha; x_1, x_2, \dots, x_l) \in \C\]
be the Gromov-Witten invariant,  
here $\alpha \in H_2(X; \Z)$ and $x_i \   (i=1,2, \dots, l)$ are elements of $\coh{X}{\C}.$

This is defined by using the moduli space of genus zero stable holomorphic maps with $l$ marked points and of homology class $\alpha$.
Set \[\gw_l(x_1, x_2, \dots, x_l):=
\sum_\alpha T^{\frac{\omega(\alpha)}{2 \pi}} \cdot \gw (\alpha; x_1, x_2, \dots, x_l).\]
We extend $\gw_l$ linearly and $\gw_l$ gives a morphism from $\coh{X}{\Lambda}^{\otimes l}$ to $\Lambda.$
We define the quantum product $x_1*x_2$ by the following equation:
\begin{align*}
\pd{x_1*x_2}{x_3}{X}:= \gw_3(x_1, x_2, x_3).
\end{align*}
The quantum product makes $\coh{X}{\Lambda}$ into a $\Z/2\Z$-graded commutative ring with a unit $1_X.$
$\coh{X}{\Lambda}$ equipped with this ring structure is called the quantum cohomology
and denoted by $\qcoh.$


\section{Fukaya categories and open-closed maps}\label{sec4}
\subsection{Assumptions on Fukaya categories and open closed maps}\label{suboc}
In \S \ref{suboc}, we recall some expected properties of Fukaya categories.
The construction of Fukaya categories with these properties is announced by \cite{afooo} 
(see also \cite{ganmir}). 
Let $\LL$ be a finite set of weakly unobstructed compact oriented Lagrangian submanifolds with spin structures.
We assume that $L\pitchfork L'$ for $L \neq L' \in \LL$. 
\begin{assumption}\label{many}
There exists an $n$-cyclic $\ainf$ category $\fuk{\LL}$ which satisfies the following:
\begin{enumerate}
\item $\ob\fuk{\LL} = \{ (L, b) \mid L \in \LL, \ b \ \text{is a Maurer-Cartan element of } L\}.$ \\
\item $\Hom{}{(L_1, b_1)}{(L_2, b_2)}=
    \begin{cases}
    \coh{L_1}{\mathscr{H}om(\rho_1, \rho_2)} \otimes \Lambda
       &\text{if}\  L_1=L_2, \ W_{L_1, b_1}=W_{L_2, b_2}. \\
    \displaystyle{\bigoplus_{p \in L_1 \cap L_2}}\Lambda\langle p \rangle
       &\text{if} \ L_1 \neq L_2,\  W_{L_1, b_1}=W_{L_2, b_2}. \\
     0 &\text{if otherwise},  
\end{cases}$\\
here $\mathscr{H}om(\rho_1, \rho_2)$ is the local system on $L_1$
corresponding to the group homomorphism $\rho_2 \cdot \rho_1^{-1}$. 
Note that $p \in L_1 \cap L_2$ is equipped with a $\Z/2\Z$-grading .\\
\item For $L\in \LL$ and a group homomorphism $\rho$, $\fuk{\LL}$ induces an $n$-cyclic $\ainf$ structure on the set of Maurer-Cartan elements $(\rho, b_+)$ of $L$.
We assume this $n$-cyclic $\ainf$ category is equal to $\fuk{L, \rho}$ introduced in Theorem $\ref{onelag}$. 
\end{enumerate}
\end{assumption}
\begin{remark}
In \cite{shefuk}, Sheridan constructs Fukaya categories for compact monotone symplectic manifolds.
Sheridan also shows that the monotone Fukaya category is naturally equipped with a weak proper Calabi-Yau structure.
But this monotone Fukaya category is not cyclic.
For cyclicity, see also \cite{gancyc}.
\end{remark}
The cohomology $H^\bullet \Hom{}{(L_1,b_1)}{(L_2, b_2)}$
is denoted by $HF^\bullet\big( (L_1, b_1), (L_2, b_2) \big)$.
Moreover, we assume the following: 
\begin{assumption}[see also \cite{fooo1}, \cite{fooospe}, \cite{foootoricmirror}]\label{oc}
There exists a closed open map 
\[\q: \qcoh \to HH^\bullet(\fuk{\LL}).\] 
We define 
\[\p : HH_\bullet(\fuk{\LL}) \to QH^{\bullet+n}(X)\] 
by the equation
\begin{align}\label{defoc}
\pd{x}{\p(\X)}{X}=(-1)^{n \cdot |\X|}\cyc{\q(x)}{\X},
\end{align}
here $x \in \qcoh$ and $\X \in HH_\bullet(\fuk{\LL}).$
They satisfy the following conditions:
\begin{enumerate}
\item $\q$ is a unital ring homomorphism.
\item 
         For $\eta \in H^2_c(X \setminus L; \C)$, we assume
         \[ [r_{(L, b)}] \circ \q(\eta)=i^*_{L, b}(\eta),\]
         here $\eta$ is naturally considered as an element of $\qcoh.$
\item For $\X, \X' \in HH_\bullet(\fuk{\LL}),$
         we assume
         \[\pd{\p(\X)}{\p(\X')}{X}=(-1)^{\frac{n(n+1)}{2}+ ?} \muk{\X}{\X'},\]
         here the sign $?$ only depends on parity of $n$ and $|\X|$.        
\item  Let $[L] \in H^n(X; \Z)$ be the Poincar\'e dual of the homology class of $L$, i.e.,
          $[L]$ is defined by the equation $\int_L x|_L=\int_X [L] \cup x$.
          We consider $1_L \in HF^\bullet(L, b)$ as an element of $HH_\bullet(\fuk{\LL}).$ 
          We assume \[\p(1_L)=[L].\]  
\end{enumerate} 
\end{assumption}
\begin{remark}
There is many results closely related to this assumption.
See, e.g., \cite[Appendix B]{ganaut}, \cite[Propositions 2.1, 2.2, 2.6, Lemmas 2.3, 2.7, 2.14]{shefuk}
for monotone cases $($see also \cite{ritmon}$)$.
For toric cases, see, e.g., \cite[\S 4.7]{foootoricmirror}. 
\end{remark}
By Equation (\ref{modcyc}), we see that $\p$ is a $\qcoh$-module map, i.e.,
         \begin{align}\label{modoc}\p \big( \q(x) \cap \X \big)=x*\p(\X)\end{align}
         for  $x \in \qcoh$ and $\X \in HH_\bullet(\fuk{\LL}).$
To simplify notation, $[Z]$ is also denoted by $Z.$
As a direct consequence of Assumption \ref{oc} (3), we see that 
\begin{align}\label{Z}
\q \circ \p (\X) = (-1)^{\frac{n(n+1)}{2}+?'}Z(\X),
\end{align}
here $?'=?+n\cdot|\X|$.
\begin{lemma}\label{3.10}
Let $\A \subset \fuk{\LL}$ be a subcategory.
Suppose that $H^\bullet(\A)$ is not equivalent to the zero category.
Then $\p \circ [c_\A] \neq 0.$
\end{lemma}
\begin{proof}
Assume $\p \circ [c_\A]=0.$
By Equation (\ref{dual}) and the definition of $\p$, 
we have \[\pd{1_X}{\p \circ [c_\A] (\X)}{X}=(-1)^{n\cdot |\X|} \cyc{[r_\A] \circ \q(1_X)}{\X}\]
for all $\X \in HH_\bullet(\A).$
From this we have $\uniho =[r_\A] \circ \q(1_X)=0,$
which implies $1_L \in HF^\bullet(L, b)$ are equal to zero for all $(L, b) \in \A.$
Thus $(L, b) \cong 0$, which contradicts the assumption $H^\bullet(\A) \not \cong 0.$
\end{proof}

\subsection{Decomposition of Fukaya categories with respect to idempotents}\label{decomp}
In \S \ref{decomp}, we consider decomposition of Fukaya categories by using open closed maps (see \cite{afooo}, \cite{benloc}, \cite{evagen}, \cite{seiabs}).

For $\A \subset \fuk{\LL}$ and $(L, b) \in \fuk{\LL},$ 
set 
\[\p_\A:= \p \circ [c_\A], \ \q_\A:= [r_\A] \circ \q, 
   \ \p_{L, b}:=\p \circ [c_{(L, b)}], \ \q_{L, b}:=[r_{(L, b)}] \circ \q.\]
We note that $\q_{L, b}$ is also a ring homomorphism and $\p_{L, b}$ is a $\qcoh$-module map. 
Let $e$ be an idempotent of $\qcoh$, i.e., $e^2=e$. 
Note that $e$ is even degree.
Let $\Fuk(X; c)^{\LL}_e$ be the $n$-cyclic $\ainf$ subcategory with  
      \[\ob \big( \Fuk(X; c)^{\LL}_e \big):=\{ (L, b) \mid W_{L, b}=c, \ \q_{L, b}(e)=1_L \}.\]
Set \[\Fuk(X; c)^\LL:= \Fuk(X; c)^\LL_{1_X}, \ \ 
         \Fuk(X)^\LL_e:= \bigcup_{c \in \Lambda} \Fuk(X; c)^\LL_e.\]
Note that $\Fuk(X)^\LL_{1_X}= \fuk{\LL}.$
The restriction of $\p$ to $HH_\bullet(\Fuk(X)_e^\LL)$ is also denoted by $\p.$
Using Theorem \ref{spgen},
we have the following theorem which is due to Abouzaid-Fukaya-Oh-Ohta-Ono \cite {afooo}.
\begin{theorem}[\cite{afooo}]\label{spgenfuk}
Suppose that $\A \subset \Fuk(X)^\LL_e$ 
and $e \in \im (\p_\A).$
Then $\A$ split generates $\Fuk(X)^\LL_e.$
\end{theorem}
\begin{proof}
By assumption, we have $1_L \in \im ([r_{(L, b)}] \circ \q \circ \p \circ [c_\A])$ for $(L, b) \in \Fuk(X)^\LL_e.$
Combined with Equation (\ref{Z}), we have $1_L \in \im (Z_{(L, b)} \circ [c_\A]).$
The theorem follows from Theorem \ref{spgen}.
\end{proof}
We show some basic properties of $\Fuk(X)^\LL_e.$
\begin{lemma}
Let $e_1, e_2$ be idempotents of $\qcoh$ and $(L_i, b_i) \in \Fuk(X)^\LL_{e_i} \ (i=1, 2).$  
Suppose that $e_1*e_2=0$.
Then $HF^\bullet\big((L_1, b_1), (L_2, b_2)\big)=0.$
\end{lemma}
\begin{proof}
For $x \in HF^\bullet\big((L_1, b_1), (L_2, b_2)\big)$, we have 
$x=1_{L_1} \cdot x \cdot 1_{L_2}= \q_{L_1, b_1}(e_1) \cdot x \cdot  \q_{L_2, b_2}(e_2).$
By Equation (\ref{commute}), we see that
\[\q_{L_1, b_1}(e_1) \cdot x \cdot \q_{L_2, b_2}(e_2)=
  x \cdot \q_{L_2, b_2}(e_2) \cdot \q_{L_2, b_2}(e_1)=x \cdot \q_{L_2, b_2}(e_2*e_1)=0.\]
Hence we have $x=0.$
\end{proof}
\begin{lemma}\label{3.11}
$\p \big( HH_\bullet(\Fuk(X)^\LL_e) \big) \subset \qcoh e$
\end{lemma}
\begin{proof}
Since $\p \big( HH_\bullet(\Fuk(X)^\LL_e) \big)$ is finite dimensional,
we can take $\X_1, \dots, \X_k \in HH_\bullet(\Fuk(X)^\LL_e)$
 such that $\p(\X_1), \dots, \p(\X_k)$ span $\p(HH_\bullet(\Fuk(X)^\LL_e).$
By definition, we have $\q_{L, b}(1-e)=0$ for $(L, b) \in \Fuk(X)^\LL_e$.
By Lemma \ref{2.2}, there exists $\varphi^N \in CC^\bullet(\Fuk(X)^\LL_e)$ such that
$M^1 (\varphi^N)=0, [\varphi^N]= \q(1-e)$ and $\varphi^N_s=0$ for $s \le N$.
By taking enough large $N$, we can assume $[\varphi^N] \cap \X_i = 0$ for $i=1, 2, \dots, k$.
Hence we have $\q(e) \cap \X_i= \X_i$. 
Since $\p$ is a module map, we see $e*\p(\X_i)= \p\big(\q(e) \cap \X_i\big)=\p(\X_i)$, 
which implies $\p(\X_i) \in \qcoh e.$
\end{proof}
The next lemma is used in \S \ref{sec5} to construct orthogonal idempotents of $\qcoh$.
\begin{lemma}\label{orth}
Let $\A_1, \A_2 \subset \fuk{\LL}.$
Suppose that $HF^\bullet\big( (L_1, b_1), (L_2, b_2) \big)=0$ for each $(L_i, b_i) \in \A_i.$
Then $\p(\X_1)*\p(\X_2)=0$ for $\X_i \in HH_\bullet(\A_i)$,
here we identify $\X_i$ with $[c_{\A_i}](\X_i)$.
\end{lemma}
\begin{proof}
We have
 \[\p(\X_1)*\p(\X_2)=\p\big( \q\circ \p(\X_1) \cap \X_2 \big)=
   (-1)^{\frac{n(n+1)}{2}+?'}\p \big( Z(\X_1) \cap \X_2\big).\]
Hence it is sufficient to show that $Z(\X_1) \cap \X_2=0.$
By Equation (\ref{dual}), we see that 
\[Z(\X_1) \cap \X_2 = [c_{\A_2}]\big([r_{\A_2}] \circ Z(\X_1) \cap \X_2\big),\]
which implies 
\[\cyc{Z(\X_1)}{\X_2}=\cyc{[r_{\A_2}]\circ Z(\X_1)}{\X_2}.\]
By Lemma \ref{2.4}, we have $\muk{\X_1}{\X'_2}=0$ for all $\X'_2 \in HH_\bullet(\A_2),$
which implies \[\cyc{[r_{\A_2}] \circ Z(\X_1)}{\X'_2}=\cyc{Z(\X_1)}{\X'_2}=\muk{\X_1}{\X'_2}=0.\]
From this, it follows that $[r_{\A_2}] \circ Z(\X_1)=0.$
Hence we have $Z(\X_1) \cap \X_2=0.$
\end{proof}

\subsection{Automatic split-generation}\label{subaut}
The next statement is a main theorem of this paper.
We note that a similar statement has independently been obtained by Ganatra \cite{ganaut}. 
\begin{theorem}[cf. \cite{ganaut}]\label{aut}
Let $\A \subset \fuk{\LL}$ be a subcategory. 
Suppose that $Z$ of $\A$ is an isomorphism.
Then there exists an idempotent $e \in QH^\bullet(X)$ such that
$\A \subset \Fuk(X)^\LL_e$ and $\A$ split generates $\Fuk(X)^\LL_e.$
Moreover, 
\[\p_\A : HH_\bullet(\A) \to QH^{\bullet+n}(X)e\]
is an isomorphism and $\q_\A$ induces an isomorphism
between $\qcoh e$ and $HH^\bullet(\A)$ by restriction.
\end{theorem}
\begin{proof}
Since $Z$ is an isomorphism, there exists $\X \in HH_\bullet(\A)$ such that
$\q_\A \circ \p_\A(\X)=\uniho$.
Set $e:=\p_\A(\X)$.
Since $\p_\A$ is a module map, we have
\[e*e=\p_\A\big( \q_\A \circ \p_\A(\X) \cap \X \big)=
  \p_\A \big( \uniho \cap \X \big) = \p_\A(\X)=e.\]
By construction, we have $\A \subset \Fuk(X)^\LL_e$ and 
$e \in \im(\p_\A).$
Combined with Theorem \ref{spgenfuk}, it follows that $\A$ split generates $\Fuk(X)^\LL_e.$\\
By Lemma \ref{3.11}, we see that $\p_\A \big( HH_\bullet(\A) \big) \subset QH^{\bullet+n}(X)e.$
Since $Z=(-1)^{\frac{n(n+1)}{2}+?'} \q_\A \circ \p_\A$ is an isomorphism, 
$\p_\A$ is injective. 
Since $e \in \p_\A \big( HH_\bullet(\A) \big)$ and $\p_\A$ is a module map, 
it follows that $\p_\A$ is a surjection to $QH^{\bullet+n}(X)e.$
Hence $\p_\A$ is an isomorphism.
Injectivity of the restriction $\q_\A|_{\qcoh e}$ follows from Equation (\ref{defoc}).
Since the dimensions of $\qcoh e$, $HH_\bullet(\A)$, and $HH^\bullet(\A)$ are the same,
it follows that $\q_\A|_{\qcoh e}$ is an isomorphism.
\end{proof}
\begin{remark}\
\begin{itemize}
\item By Lemma \ref{3.10}, we see that $e \neq 0$ if $H^\bullet(\A) \not \cong 0.$
\item If $\A$ is smooth, then the assumption of Theorem \ref{aut} is satisfied.
\end{itemize}
\end{remark}
\subsection{Eigenvalues of $c_1$}
Let $c_1$ be the first Chern class of the tangent bundle $TX$. 
$\crit$ denotes the set of eigenvalues of the operator $c_1*$. 
For $c \in \crit$, the generalized eigenspace of $c_1*$ with the eigenvalue $c$ is denoted by $\qcoh_c$
and the idempotent corresponding to $\qcoh_c$ is denoted by $e_c$ (i.e., $e_c$ is the $\qcoh_c$ summand of the unit $1_X$).
By definition, $e_c$ satisfies $\qcoh_c=\qcoh e_c.$ 

\begin{proposition}[cf. {\cite[Theorem 6.1]{aurmir}, \cite[Theorem 23.12]{fooospe}}]\label{chern}
Let  $(L, b) \in \fuk{\LL}$. 
If  $b_+ \in H^1(L; \Lambda_+),$
then $\q_{L,b}(c_1)=W_{L, b} \cdot 1_L$.
\end{proposition}
\begin{proof}
We note that the Maslov class $\mu_L \in H^2(X, L;\R) \cong H^2_c(X \setminus L;\R).$
By Assumption \ref{oc} (2), we have 
\[\q_{L, b}(c_1)=i^*_{L, b}(\mu_L/2)=
   \sum_{k \in \Z} \sum_{\mu_L(\beta)=2k} k \cdot \rho(\partial \beta) T^{\frac{\omega(\beta)}{2\pi}}
   m^\can_{s, \beta}(b_+, \dots, b_+).\]
Since $b_+ \in H^1(L; \Lambda_+)$, the degree of $m^\can_{s, \beta}(b_+, \dots, b_+)$ is $2-\mu_L(\beta).$ From this it follows that
\[ \sum_{\mu_L(\beta)=2k} \rho(\partial \beta) T^{\frac{\omega(\beta)}{2\pi}}
   m^\can_{s, \beta}(b_+, \dots, b_+)=
\begin{cases}
0 & \text{if}\ \ k\neq1 \\
W_{L,b} \cdot 1_L &\text{if} \ \ k=1,
\end{cases}\]
which proves the proposition.
\end{proof}
The next lemma gives a relationship between eigenvalues of $c_1*$ and $W_{L, b}$.
\begin{lemma}\label{4.13}
If $\q_{L,b}(c_1)=W_{L, b} \cdot 1_L,$ then
$c_1* \p_{L,b}(x)= W_{L, b} \cdot \p_{L,b}(x)$ for $x \in HF^\bullet(L,b).$
\end{lemma}
\begin{proof}
Since $\p_{L, b}$ is a module map, we have 
\[c_1* \p_{L,b}(x)= \p_{L, b}(\q_{L, b}(c_1) \cap x)=W_{L, b} \cdot \p_{L,b}(x).\]
\end{proof}
Combined with Assumption \ref{oc} (4), we have
\begin{align}\label{eigenvector}
c_1*[L]=W_{L, b} \cdot [L]
\end{align}
if $b_+ \in H^1(L; \Lambda_+).$
The categories $\Fuk(X; c)^\LL$ and $\Fuk(X)^\LL_{e_c}$ are related as follows: 
\begin{lemma}
Suppose that $(L, b) \in \Fuk(X; c)^\LL$ satisfies $\q_{L,b}(c_1)=W_{L, b} \cdot 1_L.$
Then \[(L, b) \in \Fuk(X)^\LL_{e_c}.\]
\end{lemma}
\begin{proof}
By the definition of $e_c$, there exists $\alpha \in \qcoh(1_X-e_c)$ such that $(c_1-c)*\alpha=1_X-e_c$.
By assumption, we have $\q_{L,b}(1_X-e_c)=\q_{L, b}(c_1-c) \cdot \q_{L, b}(\alpha)=0$
for all $(L, b) \in \Fuk(X; c)^\LL$.   
Hence we have $\q_{L, b}(e_c)=\q_{L, b}(1_X)=1_L$, which implies the lemma. 
\end{proof}

\subsection{Potential functions}\label{potential}
In \S \ref{potential}, we assume that $H_1(L ;\Z)$ of $L \in \LL$ is torsion free. 
Moreover we assume the following condition (\cite[Condition 6.1]{foootoricdeg}):
\begin{condition}\label{positive}
Let $L$ be a compact oriented Lagrangian submanifold. 
If $\beta \in H_2(X, L; \Z)$ satisfies $ \mathscr{M}_1(L; \beta) \neq \emptyset$ and $\beta \neq 0$,
then $\mu_L(\beta) \ge 2$.
\end{condition}
We note that this condition is satisfied if 
\begin{itemize}
\item $L$ is monotone, i.e., $\mu_L=c \cdot \omega \in H^2(X, L; \Z)$ for some $c \in \R_{>0}.$
\item $L$ is a torus fiber of a compact symplectic toric Fano manifold.
\end{itemize}
Note that if $L$ satisfies Condition \ref{positive},
then $m^\can_{0, \beta}(1) \in \Q \cdot 1_L$ and 
\begin{align}\label{eq4.2}
m^\can_{s, \beta}(b, b, \dots, b)=\frac{\langle \partial \beta, b \rangle^s}{s!} \cdot m^\can_{0, \beta}(1)
\end{align}
for $b \in H^1(L;\C)$ (see \cite[Appendix 1]{foootoricdeg}).
Define $n_\beta \in \Q$ by $n_\beta \cdot 1_L=m^\can_{0, \beta}(1).$ 
For $b_+ \in H^1(L; \Lambda_+)$,
we see that 
\[W_{L, b}= \sum_{\mu_L(\beta)=2}
   n_\beta \cdot T^{\frac{\omega(\beta)}{2\pi}} \cdot \rho(\partial \beta) \cdot
    \exp (\langle \partial \beta, b_+ \rangle)\]
by using Equation (\ref{eq4.2}) (see \cite[Theorem A.2]{foootoricdeg}).
Thus all elements of $\mathrm{Hom}\big(\pi_1(L), \C^*\big) \times H^1(L; \Lambda_+)$ are Maurer-Cartan elements. 
Let $\Lambda[H_1(L; \Z)]$ be the Laurent polynomial ring with the coefficient $\Lambda$ corresponding to the monoid
$H_1(L; \Z)$. 
The monomial corresponding to $a \in H_1(L; \Z)$ is denoted by $y^a.$
The completion of the ring $\Lambda[H_1(L; \Z)]$ with respect to the Gauss norm is denoted by 
$\Lambda \langle \langle H_1(L; \Z) \rangle \rangle$ (see \cite{bosnon}, \cite{foootoricmirror}).
Then $\Lambda \langle \langle H_1(L; \Z) \rangle \rangle$ is considered as the ring of $\Lambda$-valued functions defined on $H^1(L; \Lambda_{=0}):=\mathrm{Hom}\big(H_1(L; \Z), \Lambda_0^*\big).$
\begin{definition} We define a potential function $W_L$ of $L$ by  
\[W_L:=  \sum_{\mu_L(\beta)=2}
   n_\beta \cdot T^{\frac{\omega(\beta)}{2\pi}}\cdot y^{\partial \beta}.\]
By the Gromov compactness, this is an element of $\Lambda \langle \langle H_1(L; \Z) \rangle \rangle.$
\end{definition}

By the isomorphism 
$\mathrm{Hom}\big(\pi_1(L), \C^*\big) \cong \mathrm{Hom}\big(H_1(L; \Z), \C^*\big)$ 
and the natural inclusion $\C^* \subset \Lambda^*_0$,
we consider $\mathrm{Hom}\big(\pi_1(L), \C^*\big)$ as a subset of $H^1(L; \Lambda_{=0}).$
Then  $\mathrm{Hom}\big(\pi_1(L), \C^*\big) \times H^1(L; \Lambda_+)$ is isomorphic to $H^1(L; \Lambda_{=0})$
by sending $(\rho, b_+)$ to $\rho \cdot e^{b_+}.$
Then we see that $W_L(b)=W_{L, b}$ for $b \in H^1(L; \Lambda_{=0}).$

Take a basis $\{e_i\}_{i=1}^m$ of $H_1(L; \Z)$ and set $y_i:=y^{e_i}.$  
Let $\{e^i\} \subset H^1(L; \Z)$ be the dual basis of $\{e_i\}.$
By definition, we see 
\begin{align}
 m^\can_{1, b}(e^i)&=y_i \frac{\partial W_L}{\partial y_i}(b) \cdot 1_{(L, b)}, \\
 m^\can_{2, b}(e^i, e^j)+m^\can_{2, b}(e^j, e^i)&=
 y_iy_j\frac{\partial^2 W_L}{\partial y_i \partial y_j}(b) \cdot 1_{(L, b)}.
\end{align}
for $b \in H^1(L; \Lambda_{=0}).$
The next proposition is not used later, but it may be useful to understand Landau-Ginzburg mirror of Fano manifolds.
\begin{proposition}[cf. {\cite[Remark 3.1.3]{ggi}}]
Suppose that $H_1(X; \Z)=0$ and $c_1/r \in H^2(X; \Z)$ for $\ r \in \Z_{\ge 1}$.
Let $\beta_1, \dots, \beta_m$ be elements of $H_2(X, L; \Z)$ such that
$\{ \partial \beta_i \}$ is a $\Z$-basis of $H_1(L; \Z)$ and
let $\zeta$ be an $r$th root of unity.
Set $k_i:=\mu_L(\beta_i)/2 \in \Z$ and $y_i:=y^{\partial \beta_i}.$
Then we have \[W_L(\zeta^{k_1}y_1, \dots, \zeta^{k_m} y_m)=\zeta W_L(y_1, \dots, y_m).\]    
\end{proposition}
\begin{proof}
By definition we have 
\[W_L=\!
\sum_{\substack{\mu_L(\beta)=2 \\ \partial \beta=a_1\partial \beta_1+\cdots+a_m \partial \beta_m}}\! 
         n_\beta \cdot T^{\frac{\omega(\beta)}{2\pi}}\cdot y_1^{a_1} \cdots y_m^{a_m}.\]
Since $\mu_L(\beta)=2$, we see that
\[k_1 a_1+ \cdots +k_m a_m+c_1(\beta-a_1\beta_1-\cdots -a_m \beta_m )=1.\]
We note that $\beta-a_1\beta_1-\cdots -a_m \beta_m \in H_2(X; \Z).$
By assumption, we have 
\[c_1(\beta-a_1\beta_1-\cdots -a_m \beta_m ) \in r \Z,\]
which implies the proposition.
\end{proof}
Let $\eta \in H^2_c(X \setminus L; \C)$ and 
$b':=(\rho', b_+') \in \mathrm{Hom}\big(\pi_1(L), \C^*\big) \times H^1(L; \Lambda_+)$. 

By Equation (\ref{eq4.2}), we see that 
\[ i^*_{L,b}(\eta)=\sum_{\mu_L(\beta)=2}
   n_\beta \cdot \langle\eta, \beta \rangle \cdot T^{\frac{\omega(\beta)}{2\pi}} \cdot \rho(\partial \beta) \cdot
    \exp (\langle \partial \beta, b_+ \rangle) \cdot 1_{(L, b)}.
\]
\begin{proposition}
If $i^*_{L, b}(\eta) \neq i^*_{L, b'}(\eta)$, then  $HF^\bullet\big( (L, b_+), (L, b'_+) \big)=0$.
\end{proposition}
\begin{proof}
This easily follows from Equation (\ref{commute}) and Assumption \ref{oc} (2). 
\end{proof}


\section{Applications}\label{sec5}
\subsection{Tori}\label{Tori}
In \S \ref{Tori}, we assume $L \in \mathscr{L}, L$ satisfies Condition \ref{positive}, and $L$ is homeomorphic to the $n$-dimensional torus $\mathrm{T}^n.$ 
Let $C\ell_n$ be the Clifford algebra of an $n$ dimensional $\Lambda$-vector space with a non-degenerate quadratic form. 
\begin{lemma}
If $b=(\rho, b_+) \in \mc$ is a non-degenerate critical point of $W_L$, then 
$HF^\bullet(L, b)$ is isomorphic to the Clifford algebra $C\ell_n$ as a $\Z/2\Z$-graded ring.
\end{lemma}
\begin{proof}
The proof is the same as \cite[Theorem 3.6.2]{foootoricmirror} (see also \cite{chopro}).
\end{proof}
Since $C\ell_n$ is formal (\cite[Corollary 6.4]{shefuk}) and
its Hochschild homology is one-dimensional (see the proof of \cite[Proposition 1]{kaskun}), 
the Hochschild homology of the $\ainf$ algebra $\Hom{}{(L, b)}{(L, b)}$ is also one-dimensional.
Hence the trace map gives a canonical isomorphism
$HH_\bullet\big(\Hom{}{(L, b)}{(L, b)}\big) \cong \Lambda$.
Let $p$ be the element of $H^n(L;\Z)$ such that $\int_L p=1$.
Then $p$ gives an element of the Hochschild homology corresponding to $1 \in \Lambda.$
Since $C\ell_n$ is smooth (see \cite{dyccom}), we have $\muk{p}{p} \neq 0$ by Theorem \ref{shk}.
\begin{remark}
By easy computation, we see that $(-1)^{n(n+1)/2} \muk{p}{p}$ is equal to the determinant of the 
matrix
$\left[y_iy_j\frac{\partial^2 W_L}{\partial y_i \partial y_j}(b)\right]_{i,j=1}^{i,j=n}$
$($see \cite[Section 3.7]{foootoricmirror}$)$, which also implies $\muk{p}{p} \neq 0.$
\end{remark} 
Set 
\[e_{(L, b)}:=\frac{\p_{L,b}(p)}{\muk{p}{p}}.\]
By Proposition \ref{chern} and Lemma \ref{4.13}, it follows that $c_1*e_{(L, b)}=W_{L, b} \cdot e_{(L, b)}$.
Combined with (the proof of) Theorem \ref{aut}, we have the following statement:
\begin{proposition}\label{torus} 
$e_{(L, b)}$ is an idempotent which gives an field factor, i.e., 
$\qcoh e_{(L, b)} \cong \Lambda$.
Moreover, $e_{(L, b)}$ is an eigenvector of $c_1*$ with the eigenvalue $W_{L, b}.$
Finally, $(L, b)$ split generates $\Fuk(X)^\LL_{e_{(L, b)}}.$
\end{proposition}
To obtain orthogonal idempotents, we show the following lemma:
\begin{lemma}\label{4.3}
Let $b_i=(\rho_i, b_{i, +}) \ (i=1,2)$ be elements of $\mc$ with $W_L(b_1)=W_L(b_2)$.
Assume that $b_i$ are non-degenerate critical points of $W_L$ and $b_1 \neq b_2.$
Then $HF^\bullet\big( (L, b_1), (L, b_2)\big)=0.$
\end{lemma}
\begin{proof}
We first show that $\coh{L}{\mathscr{H}om(\rho_1, \rho_2)}=0$ if $\rho_1 \neq \rho_2.$
If $n=1,$ then this statement follows by easy computation.
The case $n>1$ is follows by the K\"{u}nneth  formula.
Hence we can assume $\rho_1=\rho_2.$
Then $\Hom{}{(L, b_1)}{(L, b_2)}=\coh{L}{\Lambda}$ and $m_1(1_L)=b_2-b_1 \neq 0.$
Thus it follows that 
\[\mathrm{dim}HF^\bullet\big( (L, b_1), (L, b_2)\big) < 2^n.\]
Since $HF^\bullet\big( (L, b_1), (L, b_2)\big)$ is a
$\Z/2\Z$-graded $HF^\bullet(L, b_1)-HF^\bullet(L, b_2)$ bimodule and
$HF^\bullet(L, b_i)$ are isomorphic to $C\ell_n$, we see that
$HF^\bullet\big( (L, b_1), (L, b_2)\big)$ is a left $C\ell_n \otimes C\ell_n^{\,\mathrm{op}}$ module, 
here $\otimes$ is the tensor product of $\Z/2\Z$-graded rings.
Hence we see $HF^\bullet\big( (L, b_1), (L, b_2)\big)$ is a left $C\ell_{2n}$ module.
Since dimensions of non-trivial irreducible representations of $C\ell_{2n}$ are $2^n$, it follows that 
$HF^\bullet\big( (L, b_1), (L, b_2)\big)=0.$
\end{proof}
Combined with Lemma \ref{orth}, we see the following:
\begin{proposition}\label{tori}
Under the same assumptions of Lemma $\ref{4.3},$ we have $e_{(L, b_1)}* e_{(L, b_2)}=0.$
\end{proposition}

\subsection{Even-dimensional spheres}\label{even sphere}
In \S \ref{even sphere}, we assume that $n$ is even and $L$ is an oriented spin rational homology sphere  equipped with the rank one trivial local system.
We first recall the following result:
\begin{theorem}[{\cite[Corollary 3.8.18]{fooo1}}]
$L$ is weakly unobstructed.
\end{theorem}
Since $H^\mathrm{odd}(L; \Lambda)=0$ , $b=0$ is a unique Maurer-Cartan element.
Since $m^\can_{1, 0}$ is odd degree, we see that $m^\can_{1, 0}=0$ and
$HF^\bullet(L, 0) \cong \coh{L}{\Lambda}$ as a $\Z/2\Z$-graded vector space.
Let $p_L$ be the element of $H^n(L; \Z)$ such that $\int_L p_L=1.$
Then $p_L \cdot p_L = \alpha_L \cdot p_L + \beta_L\cdot 1_L$ for some $\alpha_L, \beta_L \in \Lambda.$ 
Using cyclic symmetry, we have
\[\alpha_L= \int_L p_L \cdot p_L= \cyc{1_L}{m^\can_{2, 0}(p_L, p_L)}=\cyc{p_L}{p_L}=0.\] 
The monotone cases of the next proposition is obtained by Biran-Membrez (\cite[Theorem 3.3]{birlag}).
\begin{proposition}\label{cubic}
$[L]^3=4\beta_L \cdot [L]$
\end{proposition}  
\begin{proof}
By Equation (\ref{mukfor}), we see that
\[\muk{1_L}{1_L}=2, \ \muk{1_L}{p_L}=\muk{p_L}{1_L}=0, \ \muk{p_L}{p_L}=2\beta_L,\]
which implies $Z_{(L, 0)}(1_L)=2\cdot p_L$ and $\ Z_{(L, 0)}(p_L)= 2\beta_L \cdot 1_L.$
Hence it follows that 
\[[L]^2=\p_{L, 0}(1_L)*\p_{L, 0}(1_L)=(-1)^{n/2+?'}\p_{L, 0}(Z_{(L, 0)}(1_L))=(-1)^{n/2+?'}2\cdot \p_{L, 0}(p_L)\]
and 
\[ [L]^3=(-1)^{n/2+?'}2\cdot \p_{L, 0}(p_L)*\p_{L, 0}(1_L)=2\cdot\p_{L, 0}(Z_{(L, 0)}(p_L))=4 \beta_L \cdot \p_{L, 0}(1_L)=4 \beta_L \cdot [L],\]
here we use Assumption \ref{oc} (4) and Equation (\ref{modoc}).
\end{proof}
If $\beta_L \neq 0$, then we easily see that 
\begin{align}\label{eq 5.1}
e^{\pm}_L:=\pm \frac{1}{4\sqrt{\beta_L}} \cdot [L] +\frac{1}{8\beta_L} \cdot [L]^2
\end{align}
are idempotents and $e^+_L*e^-_L=0.$
Moreover, we have $HF^\bullet(L, 0) \cong \Lambda \times \Lambda$ as a $\Z/2\Z$-graded ring.
Since $\Lambda \times \Lambda$ is formal and smooth, we have the following:
\begin{proposition}\label{sphere}
If $\beta_L \neq 0,$ then $e^{\pm}_L$ give field factors and $e^+_L*e^-_L=0.$
Eigenvalues of these idempotents with respect to the operator $c_1*$ are $W_{L, 0}.$
Moreover $(L, 0)$ split generates $\Fuk(X)^\LL_{e^+_L +e^-_L}.$
\end{proposition} 
If there exist two Lagrangian submanifolds, we have the following:
\begin{proposition}\label{spheres}
Let $L_1, L_2$ be oriented spin Lagrangian rational homology spheres equipped with the rank one trivial local systems such that
$L_1 \pitchfork L_2$.
\begin{enumerate}
\item If $[L_1] \cup [L_2] =0,$ then $[L_1]*[L_2]=0.$
\item If $[L_1] \cup [L_2] \neq 0,$ then $\beta_{L_1}=\beta_{L_2}$ and $W_{L_1, 0}=W_{L_2, 0}.$
\end{enumerate}
\end{proposition}
\begin{proof}
(1) 
Set $\q_{L_2, 0}([L_1])=a \cdot 1_{L_2}+ b \cdot p_{L_2} \ (a,b \in \Lambda).$
Note that \[\int_X \p_{L_i, 0}(x)=\pd{1_X}{ \p_{L_i, 0}(x)}{X}=(-1)^{n\cdot |x|} \cyc{1_{L_i}}{x}
                 =(-1)^{n\cdot |x|} \int_{L_i}x\]
for $x \in HF^\bullet((L_i, 0)).$
Since $\p_{L_2, 0}$ is a module map, we have $[L_1]*[L_2]=\p_{L_2, 0} \circ \q_{L_2, 0}([L_1]).$
Hence it follows that
\[\int_X[L_1] * [L_2]=\int_X \p_{L_2, 0} \circ \q_{L_2, 0}([L_1])=(-1)^{n}\int_{L_2} \q_{L_2, 0}([L_1])=(-1)^{n}b.\]

By assumption, we have $\int_X[L_1] * [L_2]=\int_X [L_1] \cup [L_2]=0,$
which implies $b=0$ and $[L_1]*[L_2]$ is proportional to $[L_2].$
Similarly, we see that $[L_1]*[L_2]$ is proportional to $[L_1]$.
Since $\int_X [L_i] \cup [L_i]=(-1)^{n/2}\cdot 2$ and $[L_1] \cup [L_2]=0,$ 
$[L_1]$ is not proportional to $[L_2].$
Thus we have $[L_1]*[L_2]=0.$\\
(2)
We first consider the case $\beta_{L_1} \neq 0, \beta_{L_2} \neq 0.$
In this case, we have four idempotents $e^\pm_{L_i} \ (i=1,2).$
By assumption, we see $[L_1]*[L_2] \neq 0.$
Hence we can assume $e^+_{L_1}=e^+_{L_2}.$
Since \[\int_X e^+_{L_i}= (-1)^{n/2} \cdot \frac{1}{4\beta_{L_i}},\] we have $\beta_{L_1}= \beta_{L_2}$.
  
We next consider the case $\beta_{L_1}=0.$
 If $\beta_{L_2} \neq 0$, then we can assume $[L_1] \cdot e^+_{L_2} \neq 0$ since $[L_1]*[L_2] \neq 0.$
Since $ e^+_{L_2}$ is a field factor, there exists a non-zero constant $a \in \Lambda$ such that 
$[L_1] \cdot e^+_{L_2}=a \cdot e^+_{L_2}.$
Hence we have $[L_1]^3 \cdot e^+_{L_2}=a^3 \cdot e^+_{L_2}.$
This contradicts $[L_1]^3=4\beta_{L_1} \cdot [L_1]=0.$
Hence we have $\beta_{L_2}=0.$

The proof for the case $\beta_{L_2}=0$ is similar.

Finally, we show $W_{L_1, 0}=W_{L_2, 0}.$
By Equation (\ref{eigenvector}), we see that $[L_i]$ is an eigenvector of $c_1*$ with the eigenvalue $W_{L_i, 0}$.  
Since $[L_1]*[L_2] \neq 0$, we have $W_{L_1, 0}=W_{L_2, 0}.$
\end{proof}


\section{Examples}\label{ex}
\subsection{Toric Fano manifolds}\label{toric}
In \S \ref{toric}, we give a simple proof of the toric generation criterion for toric Fano manifolds with Morse potential functions which is due to Ritter \cite{ritcir}. 
 
Let $N \cong \Z^n$ be a finitely generated free lattice and $M$ be the dual lattice.
Set $M_{\R}:=M \otimes \R.$  
Let $(X, \omega)$ be a compact symplectic toric Fano manifold corresponding to 
a moment polytope $P \subset M_{\R}.$ 
The normal fan of $P$ is denoted by $\Sigma.$
Let $\{ v_1, v_2, \dots, v_m\} \subset N$ be the set of primitive generators of rays of $\Sigma.$
Then 
\[ P=\left\{ u \in M_\R \mid \langle v_i, u \rangle+\lambda_i \ge 0 \ (i=1,2, \dots, m) \right\} \]
for some $\lambda_1, \lambda_2, \dots, \lambda_m \in \R.$
Set $\ell_i(u):=\langle v_i, u \rangle+\lambda_i.$
Since $X$ is smooth, we can assume that $\{v_1, v_2, \dots, v_n\}$ is an integral basis of $N.$
Each $v_i$ is uniquely written as $a_1v_1+a_2v_2+ \cdots + a_nv_n$ for some $a_1, a_2, \dots a_n \in \Z$
and set $\omega_i:=\lambda_i-a_1\lambda_1-a_2\lambda_2-\cdots a_n \lambda_n.$
We define a potential function $W_X$ by
\[W_X:=\sum_{i=1}^m T^{\omega_i} y_1^{a_1} y_2^{a_2} \cdots y_n^{a_n}.\]
This is an element of the Laurent polynomial ring
$\Lambda[N] \cong \Lambda[y_1, y_1^{-1}, \dots, y_n, y_n^{-1}].$
The Lagrangian torus fiber corresponding to $u \in \mathrm{int}P$ is denoted by $L_u,$
here $\mathrm{int}P$ is the set of interior points of $P.$
$L_u$ are equipped with the standard spin structures (\cite{chohol}, \cite{choflo}).
We note that $L_u$ satisfy Condition \ref{positive}.
By the torus action, $H_1(L_u; \Z)$ is naturally identified with $N$
and hence $\Lambda[N]$ is naturally considered as a subset of
$\Lambda \langle \langle H_1(L_u; \Z) \rangle \rangle.$
By \cite[Theorem 4.5]{foootoric1} (see also \cite{choflo}), we have
\[W_{L_u}(y_1, y_2, \dots, y_n)= \sum_{i=1}^m T^{\ell_i(u)}y_1^{a_1} y_2^{a_2} \cdots y_n^{a_n}.\]
Therefore it follows that  
\begin{align}\label{toricpot}
W_{L_u}(y_1, \dots, y_n)=W_X(T^{\ell_1(u)}y_1, \dots, T^{\ell_n(u)}y_n).
\end{align}
Let $c=(c_1, \dots, c_n) \in \Lambda^n$ be a critical point of $W_X.$
We define $u(c) \in M_\R$ by the condition $\ell_i\big(u(c)\big)= \mathrm{val}(c_i) \ (i=1,2,\dots, n).$
By \cite[Theorem 7.8]{foootoric1}, it follows that $u(c) \in \mathrm{int}P.$
Set $b(c):=(T^{-\mathrm{val}(c_1)}c_1, T^{-\mathrm{val}(c_2)}c_2, \dots, T^{-\mathrm{val}(c_n)}c_n).$
Under the identification $H^1\big( L_{u(c)}; \Z\big) \cong M,$
$b(c)$ gives an element of $H^1\big( L_{u(c)}; \Lambda_{=0} \big).$ 
By Equation (\ref{toricpot}), $b(c)$ is a critical point of $W_{L_{u(c)}}.$
Moreover, $b(c)$ is non-degenerate if and only if $c$ is non-degenerate.
\begin{theorem}[{\cite[Theorem 6.17]{ritcir}}, see also \cite{evagen}]\label{Toric}
Suppose that $\LL$ contains all $L_{u(c)}$ such that $c$ is a critical point of $W_X.$
If $W_X$ is Morse, i.e., all critical points are non-degenerate, 
then 
\[\{\big(L_{u(c)}, b(c)\big) \mid c \ \text{is a critical point of}\  W_X\}\]
split generates $\fuk{\LL}.$ 
\end{theorem}
\begin{proof}
By \cite[Theorem 6.1]{foootoric1}, the number of non-degenerate critical points of $W_X$ is $\mathrm{dim}\qcoh.$
Hence the theorem follows from Theorem \ref{spgenfuk}, Proposition \ref{torus}, and Proposition \ref{tori}.
\end{proof}

\subsection{Blow-up of $\C P^2$ at four points}

Let $X$ be the blow-up of $\C P^2$ at 4 points of general position equipped with the standard complex structure.
Then $X$ has a unique symplectic structure
such that the cohomology class of the symplectic form $\omega$ is equal to $c_1$. 
Let $H \in H^2(X; \Z)$ be the pullback of the hyperplane class of $\C P^2$
and $E_1, E_2, E_3, E_4 \in H^2(X;\Z)$ be the Poincar\'e dual classes of the exceptional divisors.
Then we know that $c_1=3H-E_1-E_2-E_3-E_4.$  

By \cite[Theorem 1.4]{lilag},
there exist oriented Lagrangian spheres $S_1, S_2, S_3, S_4$ such that 
$[S_i]=E_i-E_{i+1} \ (i=1,2,3), \ [S_4]=H-E_1-E_2-E_3$ (see also \cite[Section 5.1]{birlag}).
Each $S_i$ is equipped with a unique spin structure.
To simplify notation, $S_i$ with the Maurer-Cartan element $0$ is also denoted by $S_i.$
 
By \cite[Theorem 4.25]{paswal}, there exists a monotone Lagrangian torus with a spin structure $L$
such that the potential function $W_L$ is given by 
\[\left((1+y_1+y_2)(1+\frac{1}{y_1})(1+\frac{1}{y_2})-3\right) \cdot T\]
with respect to an appropriate basis of $H_1(L; \Z).$
 
By direct computation, it follows that $W_L$ has three critical points $b_1, b_2, b_3$ with critical values
$\frac{5+5\sqrt{5}}{2}\cdot T, \frac{5-5\sqrt{5}}{2} \cdot T, -3\cdot T$ respectively.
Moreover these critical points are non-degenerate.
By Hamiltonian perturbation, we can assume $S_1, S_2, S_3, S_4, L$ intersect each other transversely.
We assume that $\LL$ contains these five Lagrangian submanifolds.
\begin{theorem}\label{blow-up}
Objects $S_1, S_2, S_3, S_4, (L, b_1), (L, b_2)$ split generate $\fuk{\LL}$.
Moreover, $\qcoh$ is semi-simple.
\end{theorem}
\begin{proof}
Semi-simplicity of $\qcoh$ has already been known (\cite{baysim}, \cite{craqua}), but we give another proof.
By Proposition \ref{torus} and Proposition \ref{tori}, it follows that there exist three idempotents $e_{(L, b_1)}, e_{(L, b_2)}, e_{(L, b_3)}$ such that they are orthogonal to each other and
$\qcoh e_{(L, b_i)} \cong \Lambda.$
Since $\int_X [S_1]*[S_1] \neq 0$, we see that $[S_1], [S_2], [S_3], [S_4], [S_1]*[S_1]$ are linearly independent.
By Equation (\ref{eigenvector}) these are eigenvectors of the operator $c_1*.$
We note that $[S_1]$ and $[S_1]*[S_1]$ have the same eigenvalue $W_{S_1, 0}.$
We easily see that $[S_i] \cup [S_j] \neq 0$ if and only if $|i-j| \le 1.$
By Proposition \ref{spheres}, this implies $\beta_{S_i}= \beta_{S_j}$ and $W_{S_i, 0}=W_{S_j, 0}$ for each $i, j.$
If $\beta_{S_i}=0,$ then $[S_1], [S_2], [S_3], [S_4], [S_1]*[S_1]$ are nilpotent.
Since $e_{(L, b_i)}$ are linearly independent idempotents, $e_{(L, b_1)}, e_{(L, b_2)}, e_{(L, b_3)}, [S_1], [S_2], [S_3], [S_4], [S_1]*[S_1]$ are linearly independent.
This contradicts $\mathrm{dim} \qcoh=7$.
Thus, we have $\beta_{S_i} \neq 0$.
By Proposition \ref{sphere}, we have two idempotents $e_{S_i}^{\pm}$ for each $S_i$. 
Since $W_{L, b_1}$ and $W_{L, b_2}$ are irrational and $W_{S_i, 0}$ are rational, we see that
$e_{(L, b_i)}$ is orthogonal to $e_{S_j}^\pm$ for $i=1,2, j=1,2,3,4.$
By Proposition \ref{spheres} and $[S_i] \cup [S_4]=0$ for $i=1,2$, we have $[S_1]*[S_4]=[S_2]*[S_4]=0.$
Combined with Equation (\ref{eq 5.1}), it follows that $e_{S_1}^\pm, e_{S_2}^\pm$ are orthogonal to $e_{S_4}^\pm.$
Since $[S_1] \neq [S_2],$ we can assume $e_{S_2}^+ \neq e_{S_1}^\pm.$
Thus $e_{(L, b_1)}, e_{(L, b_2)}, e_{S_1}^\pm, e_{S_2}^+, e_{S_4}^\pm$ are idempotents orthogonal to each other. 
Since $\mathrm{dim}\qcoh=7,$ this implies the semi-simplicity.
Let $\A$ be the subcategory of $\fuk{\LL}$ with objects $(L, b_1), (L, b_2), S_1, S_2, S_4$.
Then idempotents $e_{(L, b_1)}, e_{(L, b_2)}, e_{S_1}^\pm, e_{S_2}^+, e_{S_4}^\pm$ are contained in $\p(HH_\bullet(\A))$, 
which implies $1_X \in \p(HH_\bullet(\A)).$  
By Theorem \ref{spgenfuk}, it follows that $\A$ split generates $\fuk{\LL}_{1_X}=\fuk{\LL}.$
\end{proof}

\bibliographystyle{plain}

\end{document}